\documentclass[twoside,11pt]{article} 
\usepackage{jmlr2e_cmTECHREP} 

\usepackage{bm}
\usepackage{xcolor}

\usepackage{soul}

\usepackage{mathrsfs }

\newcommand{\x }{\bm{x}}

\usepackage{amsmath,amssymb,amsxtra,comment,graphicx,psfrag}
\usepackage{bm,mathrsfs}
\usepackage{mathtools}
\usepackage{xspace}
\usepackage{stmaryrd}

\usepackage{enumerate,comment,graphicx,psfrag}

\usepackage{hhline}

\usepackage{bookmark}
\usepackage{nameref}

\usepackage{verbatim}
\usepackage{tabls}

\usepackage{fancyhdr}
\usepackage{epsfig}
\usepackage{epsfig,subfigure,epstopdf}

\usepackage{pstricks,pst-plot,pst-func}
\usepackage{pspicture}
\usepackage{curves}

\usepackage{bbm}

\def\nst2{\| _*} 
 
\def\a12{A_h ^{1/2} } 
\def\d{{\mathrm d}}
\def\tr|{|\! |\! |}

\def\R {{\mathbb R}}

\def\E{{\mathcal{E}}}

\def\L{{\mathscr{L}}}

\def \x{\overline x}

\def \a{\alpha }

\def\T_h{{{\mathcal T}_h}}

\def\<{{\langle }}
\def\>{{\rangle }}

\DeclarePairedDelimiter{\norm}{\lVert}{\rVert}

\DeclareSymbolFont{matha}{OML}{txmi}{m}{it}
\DeclareMathSymbol{\varv}{\mathbf}{matha}{118}

\def\L{{\mathscr{L}}}

\newcommand\restr[2]{{
  \left.\kern-\nulldelimiterspace 
  #1 
  \vphantom{\big|} 
  \right|_{#2} 
  }}

\NewDocumentCommand{\dgal}{sO{}m}{%
  \IfBooleanTF{#1}
    {\dgalext{#3}}
    {\dgalx[#2]{#3}}%
}

\NewDocumentCommand{\dgalext}{m}{%
  \sbox0{%
    \mathsurround=0pt 
    $\left\{\vphantom{#1}\right.\kern-\nulldelimiterspace$%
  }%
  \sbox2{\{}%
  \ifdim\ht0=\ht2
    \{\kern-.625\wd2 \{#1\}\kern-.625\wd2 \}%
  \else
    \left\{\kern-.7\wd0\left\{#1\right\}\kern-.7\wd0\right\}%
  \fi
}

\NewDocumentCommand{\dgalx}{om}{%
  \sbox0{\mathsurround=0pt$#1\{$}%
  \sbox2{\{}%
  \ifdim\ht0=\ht2
    \{\kern-.625\wd2 \{#2\}\kern-.625\wd2 \}%
  \else
    \mathopen{#1\{\kern-.7\wd0 #1\{}
    #2
    \mathclose{#1\}\kern-.7\wd0 #1\}}
  \fi
}

\renewcommand{\R}{{\mathbb{R}}}

\renewcommand{\T}{\mathbb{T}^d}

\ShortHeadings{}{}
\firstpageno{1}

\begin{document}


\title
{\sc On the stability and convergence of Physics Informed Neural Networks}

 \author{ \name{\begin{center}{Dimitrios Gazoulis $^{1,2}$, \ Ioannis Gkanis $^{2}$ \and Charalambos G.\ Makridakis $^{1, 2, 3}$   }\end{center}} 
         \addr{\begin{center}{ {$^{1}$ DMAM, University of Crete, Greece \\
         $^{2 }$ IACM-FORTH, Greece
   }\\
    $^{3 }$ MPS, University of Sussex, United Kingdom} \end{center} }}


\editor{}
\maketitle

\abstract{Physics Informed Neural Networks is a numerical method which uses neural networks to approximate solutions of partial differential equations. It has received a lot of attention and is currently used in numerous physical and engineering problems. The mathematical understanding of these methods is limited, and in particular, it seems  that, a consistent notion of stability is missing. Towards addressing this issue we    consider  model problems of partial differential equations, namely   linear elliptic and parabolic PDEs. 
Motivated by tools of nonlinear calculus of variations we systematically show that coercivity of the energies and associated compactness provide a consistent framework for stability. For time discrete training we show that if these properties fail to hold then methods may become unstable. Furthermore, using tools of $\Gamma-$convergence we provide new
convergence results for weak solutions by   only requiring that the neural network spaces are chosen to have suitable approximation properties.
While our analysis is motivated by neural network–based approximation spaces, the framework developed here is applicable to any class of discrete functions satisfying the relevant approximation  properties, and hence may serve as a foundation for the broader study of  variational     nonlinear PDE solvers.}

\section{Introduction}\label{Se:1}

\subsection{PDEs and Neural Networks}\label{SSe:1.1}
In this work\footnote{V2 | V1 : 08/2023 arXiv:2308.05423} we consider  model problems of partial differential equations  (PDEs) approximated by  deep neural learning (DNN) algorithms. In particular we 
focus on linear elliptic and parabolic PDEs and Physics Informed Neural Networks, i.e., algorithms where the discretisation is based on the minimisation of the $L^2$ norm of the residual over a set of neural networks with a given architecture. 
The classical framework of numerical analysis, which evaluates algorithmic performance through notions such as stability and approximability, remains essential in this setting. Furthermore, in many scientific applications, a desirable feature of any numerical method is the preservation of fundamental qualitative properties of the underlying continuous model at the discrete level.
In a broad class of problems, stability and structural preservation are intimately related. The principal objective of this work is to introduce a new notion of stability tailored to DNN-based discretisations of PDEs and to demonstrate convergence under appropriate approximability assumptions on the neural network class, provided the training process yields stable approximations.

The integration of machine learning techniques with PDE-based modelling has generated significant activity across multiple directions: the development of neural solvers for PDEs, operator learning, 
and hybrid statistical–physical methodologies for uncertainty quantification and statistical inference. Despite notable empirical progress, foundational mathematical understanding
remains limited.

PDEs play a fundamental role in both the modelling of physical systems and the development of computational   tools in science and engineering.
Numerical solution of PDEs utilising neural networks is at an early stage and has received a lot of attention. Such methods have significantly 
different characteristics compared to more traditional methods, and have been proved quite effective, e.g.,  in solving problems in high-dimensions, or when  statistical and physical models must be combined.  
Among these approaches, Physics Informed Neural Networks (PINNs)   have emerged as one  of the most prominent frameworks, see \cite{Karniadakis:pinn:orig:2019}, \cite{Karniadakis:pinn:dxde:2021}.  Residual based methods were considered in   \cite{Lagaris_1998}, \cite{Berg_2018},   \cite{SSpiliopoulos:2018} and their references.  Other neural network methods for differential equations and related problems include, for example, \cite{kevr_1992discrete}, \cite{e2017deep}, \cite{kharazmi2019variational}, \cite{Xu}, 
\cite{chen2022deep},  \cite{georgoulis2023discrete}, 
\cite{Grohs:space_time:2023}. The term \emph{Physics Informed Neural Networks} was introduced in the highly influential paper \cite{Karniadakis:pinn:orig:2019}.
It was then used extensively in numerous physical and engineering problems;  for a broader perspective of the related methodologies and the importance of the NN methods for scientific applications, see e.g., \cite{karniadakis_kevr_2021physics}.   
 
\noindent

Nevertheless, significant theoretical questions remain unresolved. Among the key limitations is the lack of a coherent framework for analysing stability,   a crucial tool in a priori error analysis and the convergence of algorithms, 
\cite{Lax_Ricth_1956}, 
which, also provides valuable information for fixed values of the discretisation parameters, i.e., in the pre-asymptotic regime as  it is well known that unstable methods have poor algorithmic performance.
Moreover, stability is inherently problem (and algorithmic) dependent and is not always straightforward to characterise.
Towards addressing these issues, we consider representative classes of linear elliptic and parabolic PDEs, incorporating in addition, time-discrete training  for parabolic problems. The training process, while instrumental in determining the overall performance of neural network–based methods, introduces significant complexity into the analysis. As a first step, we therefore focus exclusively on time-discrete training, which allows for a more tractable yet still informative theoretical analysis.
 
Drawing on tools from nonlinear calculus of variations, we systematically demonstrate that the coercivity of energies and the associated compactness establish {an} appropriate framework for stability. In more complicated nonlinear problems, lower semicontinuity of the energies should be ensured as well.  Within this variational perspective, we interpret stability as a necessary condition that the loss functionals (interpreted as energies) must satisfy to guarantee consistent and robust convergence behaviour—particularly when the neural networks are sufficiently expressive.  For time-discrete training, we show that if these properties are not satisfied, the methods become unstable and appear not to converge. 
It is our aim in this work is to emphasise the fundamental role played by the loss functionals and their connection to classical notions of stability in Numerical Analysis. 

To formalise this, we adopt an abstract discrete approximation framework which  is consistent with the  properties of neural network spaces,  yet remains general enough to allow a  variety of discretisation schemes. 
%
Within this setting, and by employing $\Gamma$-convergence techniques, we establish new convergence results of discrete minimisers  to weak solutions without requiring additional structural assumptions on the discrete spaces or the minimisers.  

In addition to developing this  framework, we make the following contributions:

	 (i) $\ $ We demonstrate that, contrary to common beliefs in the field, residual methods like PINNs, while requiring additional smoothness at the discrete level, can still approximate solutions under low 
	 regularity assumptions.
	 
	 (ii) $\ $We reveal a novel and intriguing connection between the \emph{training} process and the typical stability concepts of classical numerical schemes.
	 
\noindent
These findings enhance both our algorithmic and mathematical understanding of this class of methods, with broader implications for the behaviour of machine learning algorithms. The proposed framework helps to \emph{explain} potential pitfalls observed in experiments and highlights the critical role of the loss functional in shaping the method's behaviour. In particular, the final loss functional (including training) is a central component of the algorithm, governing its stability and convergence.

It is important to note that the present work focuses exclusively on stability issues associated with the loss functional.
Challenges intrinsic to neural network architectures are beyond the scope of our analysis.
In more complex neural network applications —such as generative modelling or complicated PDEs— the choice of discrete spaces may also affect stability. However, in the context considered here, we require only that the discrete functions are capable of approximating functions with prescribed smoothness.
Finally, while our analysis is motivated by neural network–based approximation spaces, the framework developed here is applicable to any class of discrete functions satisfying the relevant approximation and compactness properties, and hence may serve as a foundation for the broader study of  variational neural based or not PDE solvers.

\subsection{Model problems and their  Machine Learning approximations}\label{Se:1NN}
In this work we consider linear elliptic and parabolic PDEs. To fix notation, 
we consider  simple  boundary value problems  of the form,
\begin{equation}\label{EllipticPDE}
\left \{
\begin{alignedat}{3}
&L\, u = f\quad &&\text{in}\,\,&&\varOmega \\[2pt]
&u =0 \quad &&\text{on}\,\,&&\partial\varOmega \\[2pt]
\end{alignedat}
\right .
\end{equation}
where $ u : \Omega \subset \mathbb{R}^d \rightarrow \mathbb{R} , \: \Omega $ is an open, bounded set with Lipschitz  boundary, 
$f \in L^2 (\Omega)  $ and $ L $ a self-adjoint elliptic operator of the form
$Lu := - \sum_{1 \leq i,j \leq d} \big ( a_{ij} u_{x_i } \big )_{  x_j}  +cu , $
$c>0, $ and where $ a_{ij}$ are smooth enough and 
satisfy standard ellipticity and positivity assumptions.
Further assumptions on $L$ will be discussed in the next sections. Dirichlet boundary conditions were selected for simplicity. The results of this work can be extended to other boundary conditions with appropriate technical modifications.

We shall study the corresponding parabolic problem as well. We use the compact notation 
   $ \Omega_T = \Omega \times (0,T] ,$
 $\partial \Omega_T = \partial \Omega \times (0,T] $ for some fixed time $ T>0.$
We consider the initial-boundary value problem
\begin{align}\label{ParabolicPDE}
\begin{cases}
u_t + Lu = f,  \;\;\;\; \textrm{in} \;\: \Omega_T, \\
 u =  0, \;\;\;\;\;\;\;\;\;\; \textrm{on} \;\: \partial \Omega \times (0,T] , \\
  u =  u^0, \;\;\;\;\;\;\;\;\;\; \;\;\textrm{in} \;\: \Omega \, ,
\end{cases}
\end{align}
where $ f \in L^2( \Omega_T) ,\; u^0 \in H^1_0(\Omega) $ and $ L $ is as in \eqref{EllipticPDE}.
In the sequel we shall use the compact operator notation $\L $ for either $u_t + Lu$ or $  Lu $ for the parabolic or the elliptic case correspondingly. 
The associated energies used will be the $L^2-$residuals  
\begin{equation}\label{Functional}
\mathcal{E}(v) = \int_{\Omega_D} | \L v  - f |^2 \d\overline x +\, \mu   \int_{\Omega } |  v  - u^0 |^2 \, \d  x + \tau  \, \int_{\partial \Omega  _T} |  v | ^2 \, \d \overline{ S }
\end{equation}
defined over smooth enough functions and  domains $\Omega_D$ being  $\Omega_T$ or $\Omega $ (with measures $d\overline x\ $)  for the parabolic or the elliptic case correspondingly. Clearly, the coefficient $\mu \geq 0$ of the initial condition is set to zero in the elliptic case. 
It is typical to consider regularised versions of $\mathcal{E}(v)$ as well. Such functionals have the form 
\begin{equation}\label{Functional_reg}
\mathcal{E}_{reg}(v) = \mathcal{E}(v) + \lambda \mathcal{J}  ( v )   \, , 
\end{equation}
where the regularisation parameter $\lambda=\lambda _{reg}>0$ is in principle small and $ \mathcal{J}  ( v ) $   is an appropriate functional (often a power of a semi-norm) reflecting the qualitative properties of the regularisation. 
The formulation of the method extends naturally to nonlinear versions of the generic 
operator $ \L v  - f ,$ whereby in principle both   $ \L   $ and  $f $ might depend on $v$.
 {The  loss \eqref{Functional} is a $L^2$-residual functional that imposes boundary conditions weakly. In the subsequent sections, depending on the analytical properties of the problem, we will consider alternative choices for $\mathcal{E}(v)$.
}

\subsection{Discrete Spaces generated by Neural Networks}
We consider functions 
$u_\theta$ defined through neural networks. Notice that the  structure described  is indicative and it is presented in order of fix ideas. Our results do not depend on particular neural network architectures but only on their approximation ability. 
A deep neural network maps every point $\overline x\in \varOmega _D$ to a number $u_\theta (\x) \in \R$, through
\begin{equation}\label{C_L}
	u_\theta(\x)= C_L  \circ \sigma  \circ C_{L-1} \cdots \circ\sigma \circ C_{1} (\x) \quad \forall \x\in \varOmega _D.
\end{equation}
%
The process $\mathcal{C}_L:= C_L  \circ \sigma  \circ C_{L-1} \cdots \circ\sigma \circ C_{1} $
is in principle a map $\mathcal{C}_L : \R ^m \to \R ^{m'} $; in our particular application, $m =d$ (elliptic case) or $m =d+1$ (parabolic case) and $m'=1.$  The map $\mathcal{C}_L $ is a neural network with $L$ layers and activation function $\sigma.$ Notice that to 
define $u_\theta(\x) $ for all $\x\in \varOmega _D$ we use the same $\mathcal{C}_L ,$ thus $u_\theta(\cdot ) =\mathcal{C}_L  (\cdot ) .$
Any such map $\mathcal{C}_L$ is characterised by the intermediate (hidden) layers $C_k$, which are affine maps of the form 
\begin{equation}\label{C_k}
	C_k y = W_k y +b_k, \qquad \text{where }  W_k \in \R ^ {d_{k+1}\times d_k}, b_k \in \R ^ {d_{k+1}}.
\end{equation} 
Here the dimensions $d_k$ may vary with each layer $k$ and $\sigma (y)$ denotes the vector with the same number of components as $y$,
where $\sigma (y)_i= \sigma(y_i)\, .$ 
The index $\theta$ represents collectively all the parameters of the network $\mathcal{C}_L,$ namely $W_k , b_k, $ $k=1, \dots, L .$ 
The set of all networks $\mathcal{C}_L$ with a given structure (fixed $L, d_k,  k=1, \dotsc, L\,$) of the form \eqref{C_L}, \eqref{C_k}
is called $\mathcal{N}.$ The total dimension (total number of degrees of freedom) of  $\mathcal {N} ,$ is $\dim{\mathcal {N}}= \sum _{k=1} ^L d_{k+1} (d_k +1) \, .$ 
We now define the space of functions 
\begin{equation}
	V _{\mathcal{N}}= \{ u_\theta : \varOmega _D \to \R ,  \ \text{where }  u_\theta (\x) = \mathcal{C}_L (\x), \ \text{for some } \mathcal{C}_L\in  \mathcal{N}\, \} \, .
\end{equation}
It is important to observe that $V _{\mathcal{N}}$ is not a linear space. 
We denote by $\Theta = \{ \theta \, : u_\theta \in V _{\mathcal{N}}\}.$
%
%
Clearly, $\Theta $ is a linear subspace of $ \R ^{\dim{\mathcal {N}}}.$ 

Treating boundary conditions is typically complex and requires careful consideration and   tailored approaches. To concentrate on the general methodology used in this work and avoid unnecessary technical complications, we assume that it is possible to select spaces that exactly satisfy the boundary conditions, see e.g., \cite{sukumar2022exact}.
The corresponding neural network space is then denoted by  $V _{\mathcal{N}, 0} \subset   H^2(\Omega) \cap H^1_0(\Omega).$ 
%
%
In the time-dependent case  we still use the same notation for the space  
$
V _{\mathcal{N}, 0} \subset H^1(0,T ; L^2 (\Omega)) \cap L^2(0,T ; H^2(\Omega) \cap H^1_0(\Omega)).$ 
 {Treating the boundary conditions weakly is of course possible, and in fact it is the approach taken in Section 3.2. In non-convex Lipschitz domains it is necessary to consider discrete spaces which do not belong in $ H^2(\Omega) \cap H^1_0(\Omega),$ see     Remark \ref{nonconvex_H^2}.  
Furthermore, it is noteworthy that our approach may shed light on potential imbalances in the loss function, as highlighted in Remark \ref{boundary_c} and in the analysis of Section 3.2.}

\subsection{Discrete minimisation  on  $V _{\mathcal{N}}$} Physics Informed Neural networks are based on the minimisation of residual-type functionals of the form \eqref{Functional} over the discrete set $ V _{\mathcal{N}} \, :$
\begin{definition}\label{abstract_mm_nn} Assume that the problem 
\begin{equation}\label{mm_nn:abstract}
	\min  _ {v \in  V _{\mathcal{N}} } \E (v)
\end{equation}
has a solution $v^\star \in V _{\mathcal{N}} .$ We call  $ v^\star \,  $ a deep-$ V _{\mathcal{N}}$ minimiser of $\E \, .$
\end{definition}
In the analysis of   Sections 3.1 and 4 we shall assume that $ \E (v)$ is minimised on $ V _{\mathcal{N}, 0}  \, .$ A key difficulty in studying this problem lies on the fact that $ V _{\mathcal{N}}$ is not a linear space. 
Computationally, this problem can be    equivalently formulated as a minimisation problem in $\R ^{\dim{\mathcal {N}}}$ by considering $\theta$ as the parameter vector to be identified through $\min _ { \theta \in  \Theta } \E(u_\theta),$ which   however is non-convex with respect to $\theta$ even though the functional $\E (v)$ is convex with respect to $v.$  
%
%

\subsection{Time discrete Training} To implement such a scheme we shall need   computable discrete versions of the energy $\E(u_\theta).$ 
 This can be achieved through different ways. 
  A common way to achieve this is to use appropriate quadrature for integrals over $\varOmega_D$ ({Training through quadrature}). Just to fix ideas   such a quadrature requires 
 a set $K_h$ of discrete points $z\in K_h$  and corresponding nonnegative weights $w_z$
 such that 
 \begin{equation}
\label{quadrature}
\sum _{z\in K_h} \, w_z \, g(z) \approx \int _{\Omega_D} \, g(\x) \, \d \x .
\end{equation}
  Then one can define the discrete functional %
\begin{equation}
\label{E_h}
\mathcal{E}_{Q, h}( g )  = \sum _{z\in K_h} \, w_z \,  
| \L v (z) - f(z) |^2\,  \, . 
\end{equation}
A similar treatment should be applied to the term corresponding to the initial condition $\int_{\Omega} |v - u^0|^2 dx$, and to the boundary condition, if imposed weakly.  Notice that both deterministic and probabilistic (Monte-Carlo, Quasi-Monte-Carlo)  quadrature rules are possible, yielding different final algorithms. Although our stability framework is designed to incorporate the training process, in this work we do not examine in detail the influence of quadrature (and, by extension, of training) on the stability and convergence of the algorithms. This requires a much more involved technical analysis and it will be the subject of future research. However, it will be instrumental for studying the   notion of stability introduced herein, to consider a hybrid algorithm where quadrature (and discretisation) is applied only to the time variable of the parabolic problem. This  approach is instrumental in the design and analysis of time-discrete methods for evolution problems, and we believe that it is quite useful in the present setting. 

To apply a quadrature in the time integral only we proceed as follows: Let $0=t^0< t^1< \cdots <t^N=T$ define a partition of $[0,T]$ and $I_n:=(t^{n-1},t^n],$ $k_n:=t^n-t^{n-1}.$
We shall denote by $v^m(\cdot)$ and $f^m(\cdot)$ the values $v (\cdot,t^m)$ and $f(\cdot ,t^m).$ Then we define 
the discrete in time quadrature by 
 \begin{equation}
\label{quadrature_time}
\sum _{n=1}^N \,  k_n \, g(t^{n}) \approx \int _{o}^T \, g(t) \, \d t.
\end{equation}
We proceed to define the time-discrete version of the functional \eqref{Functional} as follows
\begin{equation}\label{Functional_k}
\mathcal{G}_{k, IE} (v) = \sum _{n=1}^N \,  k_n \, \int_{\Omega }\big | \frac{v ^n-  v^{n-1}}{k_n} + L v ^n  - f^n \big |^2 \, \, \d x +\, 
|v  - u^0 |_{H^1 (\Omega)}^2 
\end{equation}
We shall study the stability and convergence properties of 
the minimisers of the problems:
\begin{equation}\label{ieE-minimize_k}
	\min  _ {v \in  V _{\mathcal{N}, 0} } \mathcal{G} _{k, IE}(v)\, . 
\end{equation}
%
%
It will be interesting to consider a seemingly similar (from the point of view of quadrature and approximation) discrete functional:
\begin{equation}\label{Functional_k_Ex_intro}
\mathcal{G}_{k, EE} (v) = \sum _{n=1}^N \,  k_n \, \int_{\Omega }\big | \frac{v ^n-  v^{n-1}}{k_n} + L v ^{n-1}  - f^{n-1} \big |^2 \, \, \d x +\, |v  - u^0 |_{H^1 (\Omega)}^2,
\end{equation}
and compare its properties to the functional $ \mathcal{G} _{k, IE} \, , $ and the corresponding $V _{\mathcal{N}}$ minimisers.

\section{Our results}
\label{sec:G_convergence}



\subsubsection*{\it Energy Stability.} 
  Our analysis   is driven  by two 
key properties which are roughly stated as follows:  The energies $\mathcal{E}_\ell ,$ defined on 
a sequence of discrete spaces $(V_\ell, \|\cdot\|_{V_\ell}),$  where $\ell$ stands for a discretisation parameter, are called \emph{Energy Stable} if the following two properties are satisfied
\begin{enumerate}
	\item [{[S1]}]  If the sequence   $\mathcal{E}_\ell [v_\ell]  $ is uniformly bounded
$$\mathcal{E}_\ell [v_\ell] \leq C,
$$  
then  there exists a constant $C_1>0,$ independent of the sequence $\{v_\ell \} $ and $\ell, $ such that 
\begin{align}
&\|v_\ell\|_{V_\ell} \le C_1 . \label{coer:dg_seminorm}
\end{align}
\item [{[S2]}]  Uniformly bounded sequences in $\|v_\ell \| _{V_\ell}$ have convergent subsequences in $H,$
\end{enumerate}
where $H$ is a normed space (typically a Sobolev space) which depends on the form of the discrete energy considered. 
Property [S1] requires that  $ \mathcal{E}_\ell [\, \cdot \, ] 
$ is coercive with respect to (possibly $\ell$-dependent) norms (or semi-norms). 
Further, [S2], implies that, although $\|\cdot \| _{V_\ell}$ are $\ell$-dependent,  they must be sufficiently strong to ensure that from any sequence uniformly bounded in these norms, one can extract a subsequence that converges in a weaker topology (induced by the space $H$). 

We argue  that these properties provide a consistent framework for stability even in the presence of training.
Although, in principle, the use of discrete norms is motivated from a nonlinear theory,   \cite{grekasPhD}, 
in order to focus on ideas rather than on technical tools, we started our study in this work on simple linear problems.

 Section 3 is devoted to elliptic problems and Section 4 to parabolic. 
 In Section 3.1 and Section 3.2 we consider the same elliptic operator but posed on 
 convex and non-convex Lipschitz domains respectively. It is interesting to compare the corresponding stability results, Propositions \ref{EquicoercivityofEdelta} and \ref{Prop:EquicoercivityofE(2)}. In both cases, given that no training is assumed, coercivity in [S1] holds with norms independent of $\ell,$
 but in the second case   in a weaker norm.  {Notice that most of the results of Section 3.1 are known, but we include them for completeness.}
   The stability result of the continuous formulation (without training) of the parabolic problem is Proposition \ref{EquicoercivityofG}. 
   In the case of time-discrete training, Proposition \ref{TD:EquicoercivityofG}, [S1] holds with an $\ell-$ dependent norm. It is interesting to observe that  a discrete maximal regularity estimate is required in the proof of Proposition \ref{TD:EquicoercivityofG}. Although we do not use previous results, it is useful  to compare to \cite{kovacs2016stable}, \cite{leykekhman2017discrete}, \cite{akrivis2022maximal}.
   The last example highlights that training is a key factor in algorithmic design, since it influences not only the accuracy, but crucially, the
stability properties of the algorithm.  In fact, in Section 4.2.5 we provide evidence that functionals 
related to time discrete training of the form \eqref{Functional_k_Ex_intro}, \eqref{Functional_k_Ex}, which fail to  satisfy the stability criteria [S1] and [S2], produce approximations with unstable behaviour.  
 
 Let us mention that for simplicity in the exposition we assume {in Sections 3.1 and 4,} that the discrete energies are defined on spaces where homogenous  Dirichlet conditions are satisfied (minimisation on  $V _{\mathcal{N}, 0} $ as  defined in Section 1.3). This is done only to avoid extra technical complications. 
 {However, in Section 3.2 we need to consider weakly enforced boundary conditions in view of the Remark \ref{nonconvex_H^2}.}  It is interesting to note, that  our analysis highlights  that the precise choice of the form of the boundary terms in the loss functional   affects how strong is the norm  in [S1], see Remark \ref{boundary_c} {and the analysis in Section 3.2.}

\subsubsection*{\it Convergence -- $\liminf - \limsup \ $ framework.}
We show convergence of the discrete minimisers to the solutions of the underlined PDE under minimal regularity assumptions. For certain cases, see Theorem \ref {Thrm:GammalimofEdelta} for example, it is possible by utilising the stability of the energies and the linearity of the problem, to show 
direct bounds for the errors and convergence. This is in particular doable in the absence of training. In the case of regularised fuctionals, or when time discrete training is considered we use   the liminf-limsup framework of De Giorgi, see Section 2.3.4 of \cite{DeGiorgi_sel_papers:2013}, and e.g., 
\cite{braides2002gamma},  used in the $\Gamma-$convergence of functionals arising in non-linear PDEs, see Theorems \ref{Thrm:GammalimofEdelta_reg}, \ref{Thrm:Gamma_reg_funct}, (regularised functionals) and Theorem \ref{Thrm:TimeD} (time-discrete training). 	These results show that stable functionals 
in the sense of [S1], [S2], yield neural network approximations converging to the  weak solutions of the PDEs, assuming only that the discrete spaces can approximate smooth functions effectively, see \eqref{w_ell_7}, \eqref{w_ell_7_hSm} and Remarks \ref{Rmk:NNapproximation},
\ref{Rmk:NNapproximation2}. 
This analytical framework combined with the stability notion introduced above provides a consistent and flexible toolbox, for analysing neural network (or more general) approximations to PDEs. It can be extended to various other, possibly nonlinear, problems. Furthermore, it provides a clear connection to PDE well posedness and  discrete stability when training is taking place. 

\subsubsection*{\it Previous works.}
Previous works on the analysis of  methods based on residual minimisation over neural network spaces for PDEs include 
   \cite{SSpiliopoulos:2018},   \cite{Karniadakis:pinn:conv:2019}, \cite{shin2020error},  \cite{hong2022priori}, \cite{Mishra_gen_err_pinn:2023}. In \cite{SSpiliopoulos:2018} convergence was established for smooth enough classical solutions of a class of nonlinear parabolic PDEs, without considering training of  the functional. Convergence results, under assumptions on the discrete minimisers or  the NN space,    when Monte-Carlo training was considered, were derived in  \cite{Karniadakis:pinn:conv:2019}, \cite{shin2020error},  \cite{hong2022priori}. In addition, in  \cite{shin2020error}, continuous stability of certain linear operators is used in the analysis.
   The results of    \cite{Mishra_gen_err_pinn:2023} were based on estimates where the bounds are dependent on the discrete minimisers and their derivatives. These bounds imply convergence only under the assumption that  these functions are uniformly bounded in appropriate Sobolev norms.
   The results in  \cite{hong2022priori}  with deterministic training, are related, in the sense that they are applicable to NN spaces where by construction high-order derivatives are uniformly bounded in appropriate norms. 
    {It is crucial to emphasise that when training and quadrature are considered in the method design and analysis, known approaches simply add the resulting (quadrature/training) error to C\' ea’s-type bounds, necessitating the uniform boundedness of the higher-order derivatives of minimisers. 
   In contrast, our approach proposes treating the scheme (including training) as a unified entity and addressing its (discrete) stability, which aligns with the typical a priori analysis approach in the field of numerical methods for differential equations.}  
%
   In  \cite{muller2020deep} $\Gamma$-convergence  was used in the analysis of  deep Ritz methods without training. In  \cite{loulakis2023new}, a similar  $\liminf - \limsup \,$ approach to this work was used in   machine learning algorithms with probabilistic training to derive convergence results for global and local discrete minimisers. 
 Conceptually related to this paper is the recent work on 
   Variational PINNs (the residuals are evaluated in a weak-variational sense), \cite{B_Canuto_P_Vpinn_quadrature_2022}, 
   where the role of quadrature was proven crucial in the analysis of the method.  For analysis of other residual based methods we refer to
   \cite{Bochev_Gunz_LS_book}, \cite{CLSchaback_2018_LeastS_coll} and their references.

\section{Elliptic problems}

{In this section, we investigate the convergence properties of approximations for the elliptic PDE \eqref{EllipticPDE}.  In Section 3.1, we assume that $\Omega$ is a convex Lipschitz domain, and the solution of the problem is sufficiently smooth. The more intricate case of a non-convex Lipschitz domain is addressed in Section 3.2, where we impose the boundary conditions weakly in the discrete functional. We briefly discuss the case where $f \in H^{-1}$ in Section 3.3. }

\subsection{Convex domains} 
{
For smooth enough $v$ now define the energy 
as follows
\begin{equation}\label{Functional_L}
\mathcal{E}(v) = \int_{\Omega} | L v  - f |^2 \, \d x 
\end{equation}
Define now the linear space $ {\mathcal{H}_{L, 0}} =\{ v\in H_0^1(\Omega ) : \ Lv \in L^2(\Omega)\, \} ,$ equipped   with the norm
$\| v\| _{ {\mathcal{H}_{L, 0}}  } = \{ \| v \| ^2 _{H^1 (\Omega )} + \|Lv \| ^2 _{L^2(\Omega )}\, \} ^{1/2}\, .$
We consider now the minimisation problem: 
\begin{equation}\label{VariationalProblem}
\min_{u \in   {\mathcal{H}_{L, 0}}   } \mathcal{E}(u)\, .
\end{equation}
 
Then  the (unique) solution of  \eqref{VariationalProblem} is the weak solution of the PDE \eqref{EllipticPDE}. 
 
 In this section we assume that if we select the networks appropriately, as we  increase their complexity  we may approximate any $w$ in $H^2(\varOmega)$ . To this end,   we select a sequence of spaces  $V _{\mathcal{N}, 0}  $ as follows: for each 
  $\ell \in \mathbb N$ we correspond a DNN space  $ V _{\mathcal{N}, 0}  ,$
  which is denoted by  $V_\ell $ with the following property: For each $w\in H_0^1(\Omega) \cap  H ^2(\Omega)$ 
  there exists a $w_\ell \in V_\ell$ such that,
    \begin{equation}\label{w_ell_7}\begin{split} 
 \|w_{\ell}-w\|_{H^2(\Omega)}  \leq \  \beta _\ell \, (w), \qquad 
 \text{and } \ \beta _\ell \, (w) \to 0,  \  \ \ell\to \infty\, .	\end{split}
\end{equation}
If in addition, $w \in H^m (\Omega )\cap H_0^1(\Omega) $ is in higher order Sobolev space
then 
  \begin{equation}\label{w_ell_7_hSm}\begin{split} 
 \|w_{\ell}-w\|_{H^2(\Omega)}  \leq \ \tilde  \beta _\ell \, \| w\| _ {H^m(\Omega)} , \qquad 
 \text{and } \ \tilde \beta _\ell \, \to 0,  \  \ \ell\to \infty\, .	\end{split} \end{equation}
We do not need specific rates for $ \tilde \beta _\ell \,   ,$
but only the fact that the right-hand side of \eqref{w_ell_7_hSm} has an explicit dependence of   Sobolev norms of $w .$  
This assumption is a reasonable one in view of the available approximation results of neural network spaces, 
see for example \cite{Xu}, \cite{Dahmen_Grohs_DeVore:specialissueDNN:2022,Schwab_DNN_constr_approx:2022, Schwab_DNN_highD_analystic:2023,  Grohs_Petersen_Review:2023}, 
and their references. 

\begin{remark}\label{Rmk:NNapproximation}
Due to higher regularity needed by the loss functional one has to use smooth enough activation functions, such as $\tanh $ or ReLU$ ^k ,$  that is, $ \sigma (y)= (\max \lbrace 0,y \rbrace )^k ,$ see e.g., \cite{Xu}, \cite{de2021approximation}.
\end{remark}

\begin{remark}\label{Rmk:NNapproximation2}
The above  assumptions can be relaxed by requiring that \eqref{w_ell_7} and \eqref{w_ell_7_hSm} hold specifically for $w = u$, where $u$ is the exact solution of the PDE. (In the case where $u$ does not have the required regularity, \eqref{w_ell_7} and \eqref{w_ell_7_hSm} should hold for appropriate smooth approximations of $u$). As will become evident in the proofs below, these approximation properties are required to ensure the existence of recovery sequences. For our problems, however, the existence of a recovery sequence for 
$u$ alone is sufficient to complete the proofs.

In general, the available results so far in the NN literature do not provide enough information on specific architectures required to achieve specific bounds with rates. Since the issue of the approximation properties is an important but independent problem, we have chosen to require  minimal abstract assumptions which can be used to prove convergence. In fact our only assumption for the discrete nonlinear spaces (neural based or not) $V_\ell$ are  \eqref{w_ell_7}, \eqref{w_ell_7_hSm} and their analogs in the next sections. 
\end{remark}

{Subsequently, we analyse the case where elliptic regularity bounds are satisfied and consider the sequence of energies.}

\begin{align}\label{deltaEnergies}
\mathcal{E}_{\ell}(u_{\ell}) = \begin{cases} \mathcal{E}(u_\ell) \;\;\;,\;\; u_\ell \in V_\ell =   V _{\mathcal{N}, 0}  , 
 \\ + \infty \;\;\;\;\;, \;\; \textrm{otherwise}
\end{cases}
\end{align}
where $ V_\ell  $ are chosen to satisfy \eqref{w_ell_7}.

\subsubsection{\it Stability}

Now we have the stability  of $ \mathcal{E}_{\ell} $ as a corollary of the following result, which follows from standard elliptic regularity bounds.

\begin{proposition}[Stability/Equi-coercivity]\label{EquicoercivityofEdelta} Assume that $\Omega$ is convex. Let $ (v_\ell) $ be a sequence of functions in $ V_\ell $ such that for a constant $ C>0 $ independent of $ \ell $, it holds that 
$\mathcal{E}_\ell (v_\ell) \leq C.$
Then there exists a constant $ C_1>0 $ independent of the sequence $\{v_\ell \} $ and $\ell, $ such that
\begin{equation}\label{EquicoercivityofEdelta2}
|| v_\ell ||_{H^2(\Omega)} \leq C_1 \, .
\end{equation}
\end{proposition}

\begin{proof}
Since $ \mathcal{E}_\ell (u_\ell) \leq C $, from the definition of $ \mathcal{E}_\ell $, it holds that $ \mathcal{E} (u_\ell) \leq C .$
We have that
\begin{equation}\label{ProofofEquicoercEq1} \mathcal{E}(u_\ell) = \int_{\Omega} ( |  L u _\ell|^2 -2 f \:  L u_\ell  + |f|^2 )\, \d x \leq C\, .
\end{equation}
From H\" older's inequality we have,
  since $ f \in L^2(\Omega) ,$  
\begin{equation}\label{ProofofEquicoercEq4}
\| L u_\ell \|_{L^2(\Omega)} \leq C_1\, .
\end{equation}
Finally, since $ u_\ell |_{\partial \Omega} =0 $, by the global elliptic regularity in  $ H^2 $   
 (see Theorem 4, p.334 in \cite{Evans}) we have
\begin{equation}\label{ProofofEquicoercEq5}
|| u_\ell ||_{H^2(\Omega)} \leq C_2 (|| L u _\ell ||_{L^2(\Omega)} + || u_\ell ||_{L^2(\Omega)})
\end{equation}
where $ C_2 $ depends only on $ \Omega $ and the coefficients of $ L $. Now since $ 0 \notin \Sigma \;\: (\Sigma $ is the spectrum of $ L $), by Theorem 6 in \cite{Evans} (p.324), we have
\begin{align}\label{ProofofEquicoercEq6}
|| u_\ell ||_{L^2(\Omega)} \leq C_3 || L u _\ell||_{L^2(\Omega)}
\end{align}
where $ C_3 $ depends only on $ \Omega $ and the coefficients of $ L .  $ Thus by \eqref{ProofofEquicoercEq4}, \eqref{ProofofEquicoercEq5} and \eqref{ProofofEquicoercEq6} we conclude
\begin{equation}\label{ProofofEquicoercEq7}
|| u _\ell ||_{H^2(\Omega)} \leq \tilde{C}\, .
\end{equation}
 
\end{proof}

\begin{remark} [Boundary loss] \label{boundary_c} 
In the case where the loss is 
\begin{equation}\label{Functional_boundary}
 \int_{\Omega} | L v  - f |^2 \d  x + \tau  \, \int_{\partial \Omega  } |  v | ^2 \, \d { S }
\end{equation}
the assumption $\mathcal{E}_\ell (u_\ell) \leq C$ provides control of the $\| v\| _{L^2 ( {\partial \Omega  })} $ which is not enough to guarantee that  elliptic regularity estimates will hold up to the boundary, see e.g., \cite{brezis2010functional}, \cite{salsa2016partial}. Compare with
the results in the case of Lipschitz domains see Section 3.2.
\end{remark}

\subsubsection{\it Convergence of the minimisers}
 In this subsection, we discuss the convergence properties of the discrete minimisers. Given the regularity properties of the 
 elliptic problem and in the absence of training,  it is possible to show the following convergence result which follows from the elliptic regularity and Proposition 4. 
 
 \begin{theorem}[Estimate in $H^2$]\label{Thrm:GammalimofEdelta} Let $ \mathcal{E}_{\ell} $ be the energy functionals defined  in \eqref{deltaEnergies} and   
let $ (u_\ell) , $ $u_\ell \in V_{\ell}, $ be a sequence of minimisers of $ \mathcal{E}_\ell .$ 
Then,  if $u$ is the exact solution of \eqref{EllipticPDE}, 
\begin{equation}\label {est_final}
  \|u  -u_\ell \|_{H^2(\Omega)}  \leq C\,  \inf_{\varphi   \in  V _{\ell} } \|u-\varphi \|_{H^2(\Omega)} \, ,
\end{equation}
and furthermore, $u_\ell \rightarrow u,  \;\; \; \textrm{in} \;\: H^2(\Omega)\, .$

\end{theorem}
\begin{proof}
 Let $u \in  H^2(\Omega)  \cap  H^1_0(\Omega)  $ be the unique solution of \eqref{EllipticPDE}. 
Consider the sequence of minimisers $(u_\ell)\, .$ Obviously, 
$$ \mathcal{E}_\ell (u_\ell) \leq \mathcal{E}_\ell (v_\ell), \qquad \text{for all } v_\ell \in V_\ell\, . $$
Then, 
\begin{equation}\label{coerc1}
 \mathcal{E}_\ell (u_\ell) = \int_\Omega | L u_\ell -f |^2 =
  \int_\Omega | L (u_\ell - u) |^2 \geq \beta \|u  -u_\ell \|^2 _{H^2(\Omega)} ,
  \end{equation}
  by Proposition \ref{EquicoercivityofEdelta}, which proves the first claim. For the second,
  let $u\in  H^2(\Omega)  \cap  H^1_0(\Omega)  $ be the unique solution of \eqref{EllipticPDE}. 
Consider the sequence of minimisers $(u_\ell)\, .$ 
Obviously, 
$$ \mathcal{E}_\ell (u_\ell) \leq \mathcal{E}_\ell (\tilde u_\ell),  $$
where $\tilde u_\ell $ is the recovery sequence corresponding to $u$ by assumption \eqref{w_ell_7}.
Then  $ \tilde u_\ell \rightarrow u $ in $ H^2(\Omega) \, $ 
and 
\begin{equation}\label{prooflimsupeq2}
\mathcal{E}_\ell (\tilde u_\ell) =|| L \tilde u_\ell  -f ||^2_{L^2(\Omega)} =   || L (\tilde u_\ell  -u) || ^2_{L^2(\Omega)} \, ,
\end{equation}
and  
the proof is complete in view of \eqref{coerc1}. 
\end{proof}

In the present  smooth setting, the above proof hinges on the fact that $\mathcal{E} (u)=0$ and on the linearity of the problem. In the case of regularised functional   
\begin{equation}\label{Functional_reg}
\mathcal{E}_{reg}(v) = \mathcal{E}(v) + \lambda \mathcal{J}  ( v )   \, , 
\end{equation}
the proof is more involved.  We need certain natural assumptions on the functional $ \mathcal{J}  ( v )$ to conclude the 
convergence. 
We shall work with convex functionals  $ \mathcal{J}  ( v )  $ that are $\mathcal H $ consistent, i.e., they satisfy the properties: 
\begin{equation}\label{consistent_reg}
	\begin{split}
		(i)&\quad  \text{$\mathcal{J}  ( v ) \geq 0,$} \\
		 (ii)&\quad  
\text{$\mathcal{J}(v) \leq \liminf_{\ell \rightarrow \infty}\mathcal{J} (v_\ell) $ for all weakly convergent sequences $v_\ell  \rightharpoonup v\in \mathcal H ,$}\\   (iii)&\quad \text{$\mathcal{J}(w) = \lim_{\ell \rightarrow \infty}\mathcal{J} (w_\ell) $ for all   convergent sequences $w_\ell  \rightarrow w\in \mathcal H ,$}
	\end{split}
\end{equation} 
where $ \mathcal H$ is an appropriate Sobolev (sub)space which will be specified in each statement.

The proof of the next theorem is very similar to the (more complicated) proof of the Theorem \ref {Thrm:Gamma_reg_funct} and it is omitted.

%

\begin{theorem}[Convergence for the regularised functional]\label{Thrm:GammalimofEdelta_reg} Let $ \mathcal{E}_{reg},\; \mathcal{E}_{reg, \ell} $ be the energy functionals defined in \eqref{Functional_reg} and 
\begin{align}\label{deltaEnergies_reg}
\mathcal{E}_{reg, \ell}(u_{\ell}) = \begin{cases} \mathcal{E}_{reg}(u_\ell), \;\;\;\;\; u_\ell \in V_\ell \cap H^2(\Omega)  \cap  H^1_0(\Omega) \\ + \infty,\;\;\;\;\; \;\; \textrm{otherwise}\, .
\end{cases}
\end{align}
Assume that the convex functional $ \mathcal{J}  ( v ) $ is $H^2(\Omega)$ consistent.  
Let $ (u_\ell) , $ $u_\ell \in V_{\ell}, $ be a sequence of minimisers of $ \mathcal{E}_{reg, \ell}.,$
Then,  $u_\ell \rightarrow u^ {(\lambda )},  \;\; \; \textrm{in} \;\: H^1(\Omega)\, ,$
where $u^ {(\lambda )}$ is the exact solution of the regularised problem.
%
%
\end{theorem}

\subsection{{Non-convex Lipschitz domains}}

In this subsection we discuss the case on non-convex Lipschitz domains, i.e., elliptic regularity bounds are no longer valid, and solutions might form singularities and do not belong in general to $H^2(\Omega).$
We will see that the stability notion discussed in [S1] and [S2] is still relevant but in a weaker topology than in the previous case. 

{  Due to the technical nature of the problem, we will require additional notation. We assume that $\Omega$ is a Lipschitz domain, \cite{Grisvard_book},
and let $\mathscr{D} (\overline \Omega ),$ $W ^m_p (\overline \Omega ),$ being the spaces consisting of elements of $\mathscr{D} (\mathbb {R} ^d  ),$ $W ^m_p (\mathbb {R} ^d),$ restricted to $\Omega\, ,$ \cite[Section 1.3.2]{Grisvard_book}. Here $W ^m_p (\mathcal {O}  ),$ denotes the standard  Sobolev space of (possibly non integer) order $m$ and 
   $H^m (\mathcal {O}  ) =  W ^m_2 (\mathcal {O}  ). $  Furthermore consider  the Sobolev space \( H^{1/2}(\partial \Omega) \) consisting of traces of functions in \( H^1(\Omega) \) on the boundary \( \partial \Omega \).   This is a Hilbert space, see \cite{H_Guide_FractionalSS} and its references,  with 
    inner product:
\[
\langle v, w \rangle_{H^{1/2}(\partial \Omega)} = \langle v, w \rangle_{L^2(\partial \Omega)} + \int_{\partial \Omega} \int_{\partial \Omega} \frac{(v(x) - v(y))(w(x) - w(y))}{|x - y|^{d}} \d S_x \d S_y.
\]
The associated norm is:
\[
\| v \|_{H^{1/2}(\partial \Omega)}^2 = \| v \|_{L^2(\partial \Omega)}^2 + \int_{\partial \Omega} \int_{\partial \Omega} \frac{|v(x) - v(y)|^2}{|x - y|^{d}} \d S_x \d S_y.
\]
}

\noindent
For smooth enough $v$ now define the energy 
as follows
\begin{equation}\label{Functional_L_w}
\mathcal{E}_{w, 1/2}(v) = \int_{\Omega} | L v  - f |^2 \, \d x  + \| v \|_{H^{1/2}(\partial \Omega)}^2\, . 
\end{equation}
We refer to Remark \ref{L2_boundary_loss} for the case of the weaker loss 
\begin{equation}\label{Functional_L_w}
\mathcal{E}_{w,0}(v) = \int_{\Omega} | L v  - f |^2 \, \d x  + \| v \|_{L^2(\partial \Omega)}^2\, . 
\end{equation}
Similarly to the previous section, we define the linear space $\mathcal{H}_{L} =\{ v\in H^1(\Omega ) : \ Lv \in L^2(\Omega)\, \} ,$ equipped   with the norm
$\| v\| _{ \mathcal{H}_{L}  } = \{ \| v \| ^2 _{H^1 (\Omega )} + \|Lv \| ^2 _{L^2(\Omega )}\, \} ^{1/2}\, ;$
see \cite[Section 1.5.3]{Grisvard_book}. 
We consider now the minimisation problem: 
\begin{equation}\label{VariationalProblem_w}
\min_{u \in  \mathcal{H}_{L}   } \mathcal{E}_{w, 1/2}(u)\, .
\end{equation}
 
The solution of \eqref{EllipticPDE} is clearly the unique minimiser of 
\eqref{VariationalProblem_w}. 
Furthermore,  the Euler-Lagrange equations for \eqref{VariationalProblem_w} are 
\begin{equation}\label{VariationalProblem_2_w}
 \int_{\Omega} ( L u - f ) \, Lv \,  \, \d x 
 +
\langle u, v \rangle_{H^{1/2}(\partial \Omega)}  
 =0 \qquad \text{for all } {v \in  {\mathcal{H}_{L}}   } \, . 
\end{equation}
One may show that \eqref{VariationalProblem_2_w} admits a unique solution which satisfies  \eqref{EllipticPDE}.
%

{In the analysis below we shall use the  bilinear form associated to the elliptic operator $ L ,$ ignoring the boundary data    $ B : H^1 (\Omega) \times H^1 (\Omega) \rightarrow \mathbb{R} . $ In particular,
\begin{equation}\label{BilinearForm}
\begin{gathered}
B(u,v) = \int_\Omega \Big (\, \sum_{i,j=1}^d a_{ij}u_{x_i}v_{x_j}  +cuv \: \Big ) \, \d x \, .
\end{gathered}
\end{equation}
In the sequel, we shall assume that the coefficients  $ a_{ij},$ $ c $ are smooth enough and satisfy the required positivity properties for our purposes.

In this section, we shall assume the analog of \eqref{w_ell_7}, \eqref{w_ell_7_hSm} without imposing zero boundary conditions on the discrete neural network spaces. Otherwise, the assumption is entirely similar: 
%
%
%
for each 
  $\ell \in \mathbb N$ we correspond a DNN space  $ V _{\mathcal{N}}  ,$
  which is denoted by  $V_\ell $ with the following property: For each $w\in   H ^2(\Omega)$ 
  there exists a $w_\ell \in V_\ell$ such that,
    \begin{equation}\label{w_ell_7_w}\begin{split} 
 \|w_{\ell}-w\|_{H^2(\Omega)}  \leq \  \beta _\ell \, (w), \qquad 
 \text{and } \ \beta _\ell \, (w) \to 0,  \  \ \ell\to \infty\, .	\end{split}
\end{equation}
If in addition, $w \in H^m (\Omega ) $ is an element of a  in higher order Sobolev space
we have
  \begin{equation}\label{w_ell_7_hSm_w}\begin{split} 
 \|w_{\ell}-w\|_{H^2(\Omega)}  \leq \ \tilde  \beta _\ell \, \| w\| _ {H^m(\Omega)} , \qquad 
 \text{and } \ \tilde \beta _\ell \, \to 0,  \  \ \ell\to \infty\, .	\end{split} \end{equation}

The discrete loss is defined as 
\begin{align}\label{deltaEnergies_w}
\mathcal{E}_{w, \ell}(u_{\ell}) = \begin{cases} \mathcal{E}_{w, 1/2} (u_\ell)\, ,  \;\;\;\;\; u_\ell \in V_\ell =   V _{\mathcal{N}}  , 
 \\ + \infty \, ,\;\;\;\;\; \;\; \textrm{otherwise,}
\end{cases}
\end{align}
where $ V_\ell  $   satisfy \eqref{w_ell_7_w}.
We have the following stability result:

\begin{proposition}\label{Prop:EquicoercivityofE(2)} The functional $ \mathcal{E}_{w, \ell} $ defined in \eqref{deltaEnergies_w}  is stable with respect to the $ H^1 $-norm: Let $ (u_\ell) $ be a sequence of functions in $ V_\ell $ such that for a constant $ C>0 $ independent of $ \ell $, it holds that
\begin{equation}\label{EquicoercivityofEdelta1(2)}
\mathcal{E}_{w, \ell} (u_\ell) \leq C.
\end{equation} 
Then there exists a constant $ C_1>0 $ independent of the sequence $\{u_\ell \} $ and $\ell, $ such that 
\begin{equation}\label{EquicoercivityofEdelta2(2)}
\| u_\ell \|_{H^1(\Omega)} \leq C_1\, .
\end{equation}
\end{proposition}

\begin{proof}
{Assume that $\mathcal{E}_{w, \ell} (u_\ell) \leq C.$ Then $ \| L u_\ell   \|_{L^2( \Omega)}     + \| u_\ell \|_{H^{1/2}(\partial \Omega)} \leq C_1\, ,$
are uniformly bounded. Standard regularity results for the elliptic problem and the trace theorem, \cite[Chapter 1] {Grisvard_book}, imply 
\begin{equation}\label{stab_w_1}
 \|u_\ell ||_{H^1(\Omega)}  \leq C_{reg} \left [ \| L u_\ell   \|_{L^2( \Omega)}     + \| u_\ell \|_{H^{1/2} (\partial \Omega)} \right ]\leq C_{reg}  C_1\, ,
\end{equation}
see e.g., \cite [Theorem 3]  {Savare_Elliptic} for a more general precise statement. The proof is then complete.
 }
%
\end{proof}

The convergence proof below relies on a crucial $\limsup$ inequality which is proved in the next Theorem  \ref{Thrm:Gamma_reg_funct}. 

\begin{theorem}[Convergence  in $H^1$]\label{Thrm:GammalimofEdelta_H1} Let $ \mathcal{E}_{w, \ell} $ be the energy functionals defined  in \eqref{deltaEnergies_w} and   
let $ (u_\ell) , $ $u_\ell \in V_{\ell}, $ be a sequence of minimisers of $ \mathcal{E}_{w, \ell}$, where $\Omega$ is a possibly non-convex Lipschitz domain. 
Then,  if $u$ is the exact solution of \eqref{EllipticPDE}, 
\begin{equation}\label{conv_H1}
u_\ell \rightarrow u,  \;\; \; \textrm{in} \;\: H^1(\Omega)\, ,
\qquad \ell \to \infty\, .
\end{equation}

\end{theorem}
\begin{proof}
 {Let  $ u \in \mathcal {H} _L  $   be the unique solution of \eqref{EllipticPDE}. 
Consider the sequence of minimisers $(u_\ell)\, .$ Obviously, 
$$ \mathcal{E}_{w, \ell}  (u_\ell) \leq \mathcal{E}_{w, \ell}  (v_\ell), \qquad \text{for all } v_\ell \in V_\ell\, . $$
By the proof of Proposition \ref{Prop:EquicoercivityofE(2)}, we have, for $c_0>0,$
\begin{equation}\label{coerc1_1}\begin{split}
 \mathcal{E}_{w, \ell} (u_\ell) = & \int_\Omega | L u_\ell -f |^2  + \| u_\ell \|_{H^{1/2} (\partial \Omega)} ^2 \\
 = &
  \int_\Omega | L (u_\ell - u) |^2 +\|(u_\ell - u) \|_{H^{1/2} (\partial \Omega)} ^2 \geq c_0\|u  -u_\ell \|^2 _{H^1(\Omega)} \, .
 	\end{split} \end{equation}
  Furthermore, let $\tilde u_\ell $ be  the recovery sequence corresponding to $u$  constructed in the proof of Theorem \ref{Thrm:Gamma_reg_funct}. Since    
$ \mathcal{E}_{w, \ell} (u_\ell) \leq \mathcal{E}_{w, \ell} (\tilde u_\ell),  $
and 
$ \lim _{\ell\to \infty}\mathcal{E}_{w, \ell} (\tilde u_\ell) =
\mathcal{E} _{w, 1/2} (  u )=0,  $
the proof follows.} \end{proof}

Next, 
we prove that the sequence of discrete minimisers $ (u_\ell) $ of the regularised functionals converges   to a global minimiser of the continuous regularised  functional.

\begin{theorem}[Convergence of the regularised functionals ]\label{Thrm:Gamma_reg_funct} Let $ \mathcal{E}_{w, reg},\; \mathcal{E}_{w, reg, \ell} $ be the energy functionals defined by
\begin{align}\label{deltaEnergies_reg_w}
\mathcal{E}_{w, reg}(v) = \mathcal{E}_{w, 1/2} (v) + \lambda \mathcal{J}  ( v )   \, ,  \qquad \mathcal{E}_{w, reg,  \ell}(u_{\ell}) = \begin{cases} \mathcal{E}_{w, reg}(u_\ell), \;\;\;\;\; u_\ell \in V_\ell \, ,  \\ + \infty,\;\;\;\;\; \;\; \textrm{otherwise}\, , 
\end{cases}
\end{align}
 respectively, where $\Omega$ is a possibly non-convex Lipschitz domain. Assume that the convex functional $ \mathcal{J}  ( v ) $ is $\mathcal{H}_L$ consistent.  
Let $ (u_\ell) , $ $u_\ell \in V_{\ell}, $ be a sequence of minimisers of $ \mathcal{E}_{w, reg,\ell}. $ 
Then,  
\begin{equation}\label{ConvOfDiscrMinE}
u_\ell \rightarrow u^ {(\lambda )},  \;\; \; \textrm{in} \;\: L^2 (\Omega), 
\quad 
  u_\ell \rightharpoonup u^ {(\lambda )}  \, ,  \;\; \; \textrm{in} \;\: H^1 (\Omega),
\qquad \ell \to \infty\, .
\end{equation}
where $u^ {(\lambda )}$ is the exact solution of the regularised problem
\begin{equation}\label{ConvOfDiscrMinE_reg_2}
\mathcal{E}_{w, reg} (u^ {(\lambda )}) = \min_{v \in \mathcal{H}_L(\Omega)} \mathcal{E}_{w, reg} (v)\, .
\end{equation}
\end{theorem}
\begin{proof}
We start with a $\liminf$ inequality: We assume there is a sequence, still denoted by $ v_\ell  $, such that $ \mathcal{E}_{w, reg,  \ell} (v_\ell ) \leq C $ uniformly in $ \ell $, otherwise  $ \mathcal{E}_{w, reg }(v) \leq \liminf_{\ell \rightarrow \infty}\mathcal{E}_{w, reg,  \ell}  (v_\ell ) = + \infty . $
The above stability result, Proposition \ref{Prop:EquicoercivityofE(2)},    implies that $ \| v_\ell  \|_{H^1(\Omega)}, \|   v_\ell    \|_{H^{1/2}(\partial \Omega)} $ are uniformly bounded.  Therefore, up to subsequences,  there exists a $v\in H^1(\Omega), $ such that $v_\ell  \rightharpoonup v $ in $ H^1 $ and $ v_\ell  \rightarrow v $ in $ L^2 ;$
furthermore, $v_\ell  \rightharpoonup v $ in $ H^{1/2} (\partial \Omega) $ and $ v_\ell  \rightarrow v $ in $ L^2 (\partial \Omega).$
Thus, 
\begin{equation}\label{lowersemicontinuity_boundary}
\| v \|_{H^{1/2}(\partial \Omega)} ^2   \leq \liminf_{\ell \rightarrow \infty} \| v_\ell  \|_{H^{1/2}(\partial \Omega)}^2\, .  
\end{equation}
Also, from the energy bound we have that $ || Lv_\ell  ||_{L^2(\Omega)} \leq C $ and therefore $ Lv_\ell  \rightharpoonup w $. Next we shall   show that $ w = Lv . $
Indeed, we have
\begin{equation}\label{LiminfIneq(2)Eq1}
\begin{gathered}
\lim_{\ell \rightarrow \infty} \int_{\Omega} Lv_\ell  \phi \, \d x = \int_{\Omega} w \phi \, \d x \;\;\;,\; \forall \; \phi \in C^\infty _0(\Omega) \, ,
\end{gathered}
\end{equation}
and  since $v_\ell  \in V_\ell\, , $
\begin{equation}\label{LiminfIneq(2)Eq2}
\begin{gathered}
\lim_{\ell \rightarrow \infty} \int_{\Omega} Lv_\ell  \phi \, \d x = \lim_{\ell \rightarrow \infty} B(v_\ell , \phi) = B(v, \phi) , \;\;\;\; \textrm{since} \;\: v_\ell  \rightharpoonup \: v \;\: \textrm{in} \;\: H^1(\Omega) \, ,
\end{gathered}
\end{equation}
hence, 
\begin{equation}\label{LiminfIneq(2)Eq3}
\begin{gathered}
B(v, \phi) = \int_{\Omega} w\phi \, \, \d x ,
\end{gathered}
\end{equation}
for all test functions  $\phi \in C^\infty _0(\Omega) .$ That is, $ Lv =w $ in the sense of distributions.
 The convexity of $ \int_\Omega | L v_\ell   -f |^2 $ implies weak lower semicontinuity, that is
\begin{equation}\label{lowersemicontinuity}
\int_\Omega | L v -f |^2 \leq \liminf_{\ell \rightarrow \infty}  \int_\Omega | L v_\ell   -f |^2\, .
\end{equation}
Therefore, in view of \eqref{lowersemicontinuity_boundary}, and since $ \mathcal{J}(v) $ is $\mathcal{H}_L$ consistent, (ii) of \eqref {consistent_reg}
implies that $ \mathcal{E}_{reg} (v) \leq \liminf_{\ell \rightarrow \infty} \mathcal{E}_{reg ,\ell} (v_\ell) $ for each such sequence  $ (v_\ell ).$

Let $ w \in \mathcal {H} _L  $ be arbitrary;  we will show the existence of a recovery sequence  $(w_\ell)$, such that  $ \mathcal{E}_{w, reg}(w) = \lim_{\ell \rightarrow \infty} \mathcal{E}_{w, reg,  \ell} (w_\ell)  .$ We establish first its existence for  $\mathcal{E}_{w,    \ell} $ and $\mathcal{E}_{w, 1/2}.  $ To this end, 
 for each $\delta >0$ we can select a smooth function $w_\delta  \in  \mathscr{D} (\overline \Omega ),$ such that   
 \begin{equation}\label{recov_bound}
\begin{split}
	\|&w -w_\delta\|_{H^1( \Omega ) } + \|Lw -Lw_\delta\|_{L^2( \Omega ) }\lesssim \delta \, , \quad \text{and,}\\
 |& w_\delta   |_{  H^s(\Omega)} \lesssim \frac{1}{\delta^s} \, .
	\end{split}
\end{equation}
Indeed, \cite[Lemma 1.5.3.9] {Grisvard_book}, implies that $\mathscr{D} (\overline \Omega )$ is dense in $ \mathcal {H} _L . $ Therefore there exists $\hat w_\delta \in \mathscr{D} (\overline \Omega )$ $\delta-$close to $w$ in the $ \mathcal {H} _L $ norm. We can then select $  w_\delta \in \mathscr{D} (\overline \Omega )$ being the the convolution of $\hat w_\delta$ with a smooth kernel such that  \eqref{recov_bound} holds.
  For $w_\delta,$ \eqref{w_ell_7_hSm_w}, there exists $ w_{\ell, \delta} \in V_\ell $ such that
$$\|w_{\ell, \delta }-w_\delta \|_{H^2(\Omega)}  \leq \ \tilde  \beta _\ell \, \| w_\delta\| _{  H^s(\Omega)}
\lesssim   \tilde  \beta _\ell \frac 1 {\delta ^s}\, 
,\qquad 
 \text{and } \ \tilde \beta _\ell \, (w) \to 0,  \  \ \ell\to \infty\, .$$
Choosing $\delta $ appropriately as function of $\tilde \beta _\ell$ we can ensure that $w_\ell =w_{\ell, \delta}$ 
satisfies, 
\begin{equation}\label{prooflimsupeq2}\begin{split}
&\| L w_\ell  -f \|_{L^2(\Omega)} \rightarrow \|Lw  -f \|_{L^2(\Omega)}\, ,\\
&\|  w_\ell  \|_{H^{1/2}(\partial \Omega)} \rightarrow \|  w   \|_{H^{1/2}(\partial \Omega)} \, ,
\end{split}
\end{equation}
where to show the last convergence we use  standard trace inequalities.  Therefore,  $ \mathcal{E}_{w, 1/2}(w) = \lim_{\ell \rightarrow \infty} \mathcal{E}_{w,  \ell} (w_\ell)  .$
Furthermore, since $ \mathcal{J}(v) $ is $\mathcal{H}_L$ consistent, (iii) of \eqref {consistent_reg}
implies that $ \mathcal{J} (w_\ell) \rightarrow   \mathcal{J}_{} (w) $ and hence
\begin{equation}\label{prooflimsupeq3}
\mathcal{E}_{reg ,\ell} (w_\ell) \rightarrow \mathcal{E}_{reg}(w).
\end{equation}
Next, let $u^ {(\lambda )}\in \mathcal {H} _L  $ be the unique solution of \eqref{ConvOfDiscrMinE_reg_2}
and consider  the sequence of the discrete minimisers $(u_\ell)\, .$ Clearly, 
$$ \mathcal{E}_{reg ,\ell} (u_\ell) \leq \mathcal{E}_{reg ,\ell} (v_\ell), \qquad \text{for all } v_\ell \in V_\ell\, . $$
In particular,  
 $ \mathcal{E}_{reg ,\ell}(u_\ell) \leq \mathcal{E}_{reg ,\ell} (\tilde u_\ell),  $ 
where $\tilde u_\ell $ is the recovery sequence constructed above corresponding to $w=u^ {(\lambda )}.$  
Thus   the discrete energies are uniformly  bounded. Then the stability result  
Proposition \ref{Prop:EquicoercivityofE(2)}, implies that
\begin{align}
 \norm{ u_\ell }_{H^1(\Omega)}   < C,  
\end{align} 
uniformly.  By the  Rellich-Kondrachov theorem, \cite{Evans},  and the $\liminf$ argument above, there exists $ \tilde   u\in \mathcal {H} _L $ 
such that $ u_\ell  \rightarrow   \tilde u$ in 
$L^2(\Omega) $ up to a subsequence not re-labeled here.
Next we show that   $ \tilde   u$ is a global minimiser of $ \E _{reg }. $ We combine  the $\liminf$ and $\limsup$
inequalities as follows: 
Let $w \in \mathcal {H} _L $, and   $ w_\ell \in V_\ell $ be its recovery sequence such that 
\eqref{prooflimsupeq3} holds.
Therefore,    the $\liminf$ inequality 
and the fact that $ u_\ell $ are   minimisers of the $\E _{reg ,\ell},$ imply that
\begin{align}
 \E _{reg } (  \tilde   u ) \le  \liminf_{\ell \rightarrow \infty }  \E _{reg ,\ell} ( u_\ell ) 
  \le  \limsup_{\ell \rightarrow \infty }  \E _{reg ,\ell} ( u_\ell ) 
\le  \limsup_{\ell \rightarrow \infty }  \E _{reg ,\ell} ( w_\ell ) 
= \E _{reg }(  w ), 
\end{align}
for all $ w \in  \mathcal {H} _L $. Therefore $ \tilde  u$ is a  minimiser of $ \E ,$ and since $u^ {(\lambda )}$ is the unique global minimiser of $ \E _{reg }$ on $\mathcal {H} _L $ we have that $\tilde  u=u^ {(\lambda )}$.

\end{proof}

\begin{remark}[Loss with $L^2$ weak boundary terms] \label{L2_boundary_loss}
A common choice in applications is to consider the loss with a weaker boundary term
\begin{equation}\label{Functional_L_w}
\mathcal{E}_{w,0}(v) = \int_{\Omega} | L v  - f |^2 \, \d x  + \| v \|_{L^2(\partial \Omega)}^2\, . 
\end{equation}
Considering the term $ \| v \|_{L^2(\partial \Omega)}^2$ is computationally convenient, but leads to a functional with weaker stability properties. However, it is  still possible to prove convergence of discrete minimisers in $L^2.$ To prove this we need a coercivity bound for 
$\mathcal{E}_{w,0}(v).$ In the smooth case such bounds are used in, e.g., \cite{zeinhofer2024unifiedframeworkerroranalysis}. In the case of Lipschitz domains and $L=-\Delta$ the a priori estimates proved by Jerison and Kenig, see  \cite{JerisonKenig_1995} and its references, can be used. To this end consider  
the discrete loss   
\begin{align}\label{deltaEnergies_w_0}
\mathcal{E}_{w, \ell}(u_{\ell}) = \begin{cases} \mathcal{E}_{w, 0} (u_\ell)\, ,  \;\;\;\;\; u_\ell \in V_\ell =   V _{\mathcal{N}}  , 
 \\ + \infty \, ,\;\;\;\;\; \;\; \textrm{otherwise,}
\end{cases}
\end{align}
where $ V_\ell  $   satisfy \eqref{w_ell_7_w}.
{
We assume   the sequence, still denoted by $ u_\ell $, satisfies   $ \mathcal{E}_{w,    \ell} (u_\ell) \leq C $ uniformly in $ \ell .$ 
As in \cite{JerisonKenig_1995} we split the problem as follows:  Consider   the functions $\tilde u _\ell$ defined as weak solutions of 
\begin{equation} 
	\begin{split}
 \label{tildeu_ell}
 L\tilde u _\ell= 0, & \qquad \text{in } \Omega,\\
  \tilde u _\ell =  u _\ell, & \qquad \text{on } \partial \Omega.
\end{split}
\end{equation}
Then $ \zeta _\ell =   u _\ell- \tilde u _\ell $ satisfies, 
\begin{equation} 
	\begin{split}
 \label{tildeu_ell}
 L \zeta _\ell = L u _\ell, & \qquad \text{in } \Omega,\\
  \zeta _\ell  =  0, & \qquad \text{on } \partial \Omega.
\end{split}
\end{equation}
Since $  u _\ell= \zeta _\ell +  \tilde u _\ell  , $ it suffices to control separately $ \zeta _\ell $ and $ \tilde u _\ell \, .$
The  stability result \eqref{stab_w_1} in Proposition \ref{Prop:EquicoercivityofE(2)}    implies that $ \| \zeta _\ell  \|_{H^1(\Omega)} $ are uniformly bounded.  
On the other hand, $\| \tilde u _\ell \| _{H^{1/2}(\Omega)} $ are uniformly bounded, see \cite{JerisonKenig_1995}[p. 165], since 
$\| u _\ell \|_{L^2(\partial \Omega)}$ are uniformly bounded. Using \cite[Theorem 1.4.3.2]{Grisvard_book} we have that $\|  u _\ell \| _{H^{1/2}(\Omega)} $ are uniformly bounded and the sequence $\{ u _\ell\} $ has a convergent subsequence in $ L^{2}(\Omega) .$ By appropriately modifying the proof of Theorem \ref{Thrm:GammalimofEdelta_H1} and utilising the above a priori bounds we can prove convergence of discrete minimisers in 
$H^{1/2}(\Omega)$ and $ L^{2}(\Omega) .$}
\end{remark}

\begin{remark}[Loss with zero boundary terms] \label{nonconvex_H^2}
A very interesting recent result by Tran, \cite{Tran2024}, suggests that the loss used in Section 3.1, see \eqref{deltaEnergies}, cannot be used to approximate singular solutions in non-convex Lipschitz domains. A modification of the argument used in \cite[Proposition 2.8]{Tran2024}
shows that any approximating sequence $u_\ell \in H^2 (\Omega ) \cap H^1_0 (\Omega ),$ satisfying $ \mathcal{E}_{ \ell} (u_\ell) \leq C $ uniformly in $ \ell ,$ cannot converge to a  $u $ which is not in $  H^2 (\Omega )\, .$ This is a result of the a priori bound \cite{Grisvard_book}[Theorem 4.3.1.4] valid for $  H^2 (\Omega )\,  $ functions and of the uniform bound in $H^1  (\Omega )\, $ 
(as  a result of $\mathcal{E}_{ \ell} (u_\ell) \leq C).$ This observation suggests that
\begin{equation}\label{VariationalProblem_Lavr_New}
\inf_{u \in   {\mathcal{H}_{L, 0}}   } \mathcal{E}(u) < \inf_{u \in   H^2 (\Omega ) \cap H^1_0 (\Omega )    } \mathcal{E}(u)\, ,
\end{equation}
which is an interesting analog of the Lavrientev gap phenomenon; for  its numerical analysis implications see,  e.g., \cite{Ball_Knowles_1987, Ball_computation_2001, Cc_review_2001}.
\end{remark}

\subsection{Forcing terms in $H^{-1}$}

Consider the problem \eqref{EllipticPDE} where the forcing term is in  $H^{-1}(\Omega).$
A conceptual limitation of residual-based methods, such as PINNs, is that the loss function is not well-defined in this setting, even though \eqref{EllipticPDE} possesses a natural variational formulation  for $f \in H^{-1}(\Omega).$ In contrast, methods like finite elements, which are based on variational formulations, can be formulated and analysed. However, a subtle implementation challenge arises: computing the action of 
$f$ on the basis functions is complex, as discussed in \cite{Kreuzer:2021aa} and similar works. As a result, approximation methods are generally required to evaluate this action. 
A natural approach  in the present setting  is to 
consider a sequence $f_\varepsilon    \in L^2 (\Omega),$ such that 
\begin{equation}
	\lim  _{\varepsilon \to 0 } f_\varepsilon  = f\, , \qquad 
	\text{in } \  H^{-1}(\Omega).\end{equation}
Thus, if $u^\varepsilon $ is the solution of \eqref{EllipticPDE} with forcing $f_\varepsilon$
we have  
\begin{equation}
	\| u - u^\varepsilon  \| _{H^{1}(\Omega)} \leq C \|  f - f_\varepsilon \|  _{H^{-1}(\Omega)} \, ,  
	\end{equation}
i.e., the distance $\| u - u^\varepsilon  \| _{H^{1}(\Omega)}$ can be done arbitrarily small. Thus the analysis of the previous sections 
implies that the minimisers of the problems
\begin{equation}\label{mm_nn:abstract}
	\min  _ {v \in  V _{\mathcal{N}, 0} } \int_{\Omega} | L v  - f_\varepsilon |^2 \, \d x 
\end{equation}
  are capable of producing accurate approximations of 
$u.$ Conceptually, this approach resembles common approximation techniques in nonlinear PDEs (such as artificial viscosity in conservation laws), where higher-gradient models with smooth solutions are employed to approximate singular solutions, guiding the design of effective numerical methods.  Residual conforming least squares methods in 
$L^2$ 
  are known to be ineffective in the finite element setting when we approximate simple linear elliptic problems. However, the minimisation problem over the neural network spaces 
$ V _{\mathcal{N}} $	
  exhibits fundamentally different characteristics. While our analysis offers valuable insights into the behaviour of these methods, further investigation is needed, including the study of more complex PDEs, other training methods,  and detailed comparisons with alternative approaches.

\section{Parabolic problems}

Let as before $ \Omega \subset \mathbb{R}^d $, open, bounded and set $ \Omega_T = \Omega \times (0,T] $ for some fixed time $ T>0. $
We consider the parabolic problem

\begin{align}\label{ParabolicPDE}
\begin{cases}
u_t + Lu = f ,\;\;\; \; \textrm{in} \;\: \Omega_T, \\
u= 0, \;\;\;\;\;\;\;\;\; \; \textrm{on} \;\: \partial \Omega \times (0,T],  \\
u = u^0, \;\;\;\;\;\;\;\;\, \textrm{on} \;\: \Omega \times \lbrace t=0 \rbrace \, .
\end{cases}
\end{align}
In  this section we discuss  convergence properties of approximations of \eqref{ParabolicPDE} obtained by minimisation of continuous and time-discrete energy functionals over appropriate sets of  neural network functions. We shall assume that $\Omega$ is a convex Lipschitz domain. The case of a non-convex domain can be treated with the appropriate modifications. 

\subsection{Exact time integrals}
So now we define $ \mathcal{G} : H^1(0,T ; L^2 (\Omega)) \cap L^2(0,T ; H^2(\Omega)  \cap  H^1_0(\Omega))\rightarrow \overline{\mathbb{R}} $ as follows
\begin{equation}\label{ParabFunctional}
\begin{gathered}
\mathcal{G}(v) =
\int_0^T \| v_t(t) + L v(t) - f(t) \|^2_{L^2(\Omega)} \d t +  | v(0) - u^0  |^2_{H^1(\Omega)}\, .
\end{gathered}
\end{equation}
We use ${H^1(\Omega)}$ seminorm for the initial condition, since then the regularity properties of the functional are better. 
Of course, one can use the ${L^2(\Omega)}$ norm instead with appropriate modifications in the proofs.

%
As before,   we select a sequence of spaces  $W _{\mathcal{N}}  $ as follows: for each 
  $\ell \in \mathbb N$ we correspond a DNN space  $ W _{\mathcal{N}}  ,$
  which is denoted by  $W_\ell $ such that: For each $w\in H^1(0,T; L^2(\Omega))\cap L^2(0,T; H^2(\Omega)) $ 
  there exists a $w_\ell \in W_\ell$ such that,
    \begin{equation}\label{w_ell_7_TD}\begin{split} 
 \|w_{\ell}-w\|_{H^1(0,T; L^2(\Omega))\cap L^2(0,T; H^2(\Omega))}  \leq \  \beta _\ell \, (w), \qquad 
 \text{and } \ \beta _\ell \, (w) \to 0,  \  \ \ell\to \infty\, .	\end{split}
\end{equation}
If in addition, $w $ has higher regularity, we assume that   \begin{equation}\label{w_ell_7_hSm_TD}\begin{split} 
 \|(w_{\ell}-w)' \|_{H^1(0,T; L^2(\Omega))\cap L^2(0,T; H^2(\Omega))}   \leq \   \tilde \beta _\ell \, \| w'\| _{ H^m(0,T; H^2(\Omega))}  , \qquad 
 \text{and } \ \tilde \beta _\ell \, \to 0,  \  \ \ell\to \infty\, .	\end{split} \end{equation}
As in the elliptic case, we do not need specific rates for $ \tilde \beta _\ell \,   ,$
but only the fact that the right-hand side of \eqref{w_ell_7_hSm_TD} has an explicit dependence of   Sobolev norms of $w .$  
 {We further assume that (possibly by employing the analog of $V_{\mathcal{N}, 0}$)  if in addition $w$ belongs in $L^2(0,T ; H^1_0(\Omega)) $ then the sequence $w_\ell$ above can be selected to belong in $W_\ell \cap L^2(0,T ; H^1_0(\Omega)) \, .$}
See \cite{Grohs_space_time_approx:2022} and its references where space-time approximation 
properties of neural network spaces are derived, see also \cite{Xu}, \cite{de2021approximation} and Remark \ref{Rmk:NNapproximation}.  
In the sequel we consider the sequence of energies
\begin{align}\label{GdeltaEnergies}
\mathcal{G}_{\ell}(u_{\ell}) = \begin{cases} \mathcal{G}(u_\ell) ,\;\;\;\;\; u_\ell \in W_\ell \cap L^2(0,T ; H^1_0(\Omega)) \\ + \infty ,\;\;\;\;\; \;\; \textrm{otherwise}
\end{cases}
\end{align}
where $ W_\ell  $ is chosen as before.

\subsubsection{Stability}

Now we have stability of $ \mathcal{G}_{\ell} $ as a corollary of the following result, which follows from standard maximal regularity estimates, Theorem 5, p.382 in \cite{Evans}.

\begin{proposition}\label{EquicoercivityofG} The functional $ \mathcal{G} $ defined in \eqref{ParabFunctional} is equicoercive with respect to the $ \\ H^1(0,T ; H^2(\Omega)  \cap  H^1_0(\Omega)) $-norm. That is,
\begin{equation}\label{G Equicoercivity}
\begin{gathered}
\textrm{If} \;\; \mathcal{G}(u) \leq C \;\; \textrm{for some} \;\: C >0 \;,\; \textrm{we have} \\
||u||_{L^2(0,T; H^2(\Omega))} + ||u'||_{L^2(0,T; L^2(\Omega))} \leq C_1
\end{gathered}
\end{equation}
\end{proposition}

\begin{proof}
As in the proof of equicoercivity for \eqref{Functional}, we have
\begin{equation}\label{Proofof G EquicoercEq1_3} \mathcal{G}(u) = \int_{\Omega_T} ( | u_t + L u|^2 -2 f \: ( u_t + L u ) + |f|^2 ) \leq C
\end{equation}
Hence, one can conclude that   since $ f \in L^2(\Omega_T) $, 
\begin{equation}\label{Proofof G EquicoercEq4}
|| u_t + L u ||_{L^2(0,T; L^2(\Omega))} \leq C_1
\end{equation}
From regularity theory for parabolic equations (see for example Theorem 5, p.382 in \cite{Evans}) we have
\begin{equation}\label{ParabRegularity}
\begin{gathered}
\textrm{ess sup}_{0 \leq t \leq T} ||u(t)||_{H^1_0(\Omega)} + ||u||_{L^2(0,T; H^2(\Omega))} + ||u'||_{L^2(0,T; L^2(\Omega))} \\ \leq \tilde{C} ( || u_t + L u ||_{L^2(0,T; L^2(\Omega))} +  {|| u(0) ||_{H^1_0(\Omega)}})
\end{gathered}
\end{equation}
the constant $ \tilde{C} $ depending only on $ \Omega \;,\: T $ and the coefficients of $ L .   $ Notice that \eqref{ParabRegularity} is a maximal parabolic regularity estimate in ${L^2(0,T; L^2(\Omega))} \, .$ This completes the proof. 
%
\end{proof}

\subsubsection{Compactness and Convergence of Discrete Minimizers}

We will also need the well-known Aubin–Lions theorem as an analog of the Rellich-Kondrachov theorem in the parabolic case, that can be found, for example, in \cite{SZ}. The proof of Theorem \ref{Thm:convergenceofdiscreteminG} is similar to the convergence in the time-discrete case below, and it is omitted.

\begin{theorem}[Aubin-Lions]
\label{Aubin-LionsTheorem} Let $ B_0,B,B_1 $ be three Banach spaces where $ B_0, B_1 $ are reflexive. Suppose that $ B_0 $ is continuously imbedded into $ B $, which is also continuously imbedded into $ B _1$, and the imbedding from $ B_0 $ into $ B $ is compact. For any given $ p_0,p_1 $ with $ 1 < p_0,p_1 < \infty $, let
\begin{align}
W = \lbrace v \: | \: v \in L^{p_0}([0,T],B_0) \;,\; v_t \in L^{p_1}([0,T], B_1) \rbrace.
\end{align}
Then the imbedding from $ W $ into $ L^{p_0}([0,T],B) $ is compact.
\end{theorem}

\begin{theorem}[Convergence of discrete   minimisers]
\label{Thm:convergenceofdiscreteminG} Let $ (u_\ell) \subset W_{\ell} $ be a sequence of minimizers of $ \mathcal{G}_\ell $, i.e.,
\begin{equation}\label{minofGdeltainWk}
\mathcal{G}_\ell (u_\ell) = \inf_{w_\ell \in W_\ell} \mathcal{G}_\ell (w_\ell)
\end{equation}
then
\begin{equation}\label{ConvOfDiscrMinG}
u_\ell \rightarrow u, \;\; \; \textrm{in} \;\: L^2(0,T; H^1(\Omega)) \, \;\: \; 
\end{equation}
where $u $ is the solution of \eqref{ParabolicPDE}.
\end{theorem}

\subsection{Time discrete training}

We shall study the stability and convergence properties of 
the minimisers of the problems, \eqref{Functional_k}:
\begin{equation}\label{ieE-minimize_k}
	\min  _ {v \in  V _{\mathcal{N}} }\mathcal{G}_{IE, k}(v)\, . 
\end{equation}
%
 Next we introduce the \emph{time reconstruction}
   $\widehat U$ 
 of a time dependent function  $U$   to be the piecewise linear approximation   
of $U$ defined by linearly 
interpolating between the nodal values $U^{n-1}$ and $U^n$: 
\begin{equation}\label{def_hatU}
\widehat U(t):=\ell_0^n(t) U^{n-1}+ \ell_1^n(t) U^{n}, \quad t\in I_n,
\end{equation}
with
$
\ell_0^n(t) := (t^n-t)/k_n$  and  $\ell_1^n(t) := (t-t^{n-1})/k_n.$ This reconstruction of the discrete solution has been proven useful in various instances,  
see \cite{baiocchi_brezzi_par_1983}, \cite{nochettoSV2000posteriori}, \cite{Dem_Stuart_Tzavaras_2001} and for higher-order versions \cite{MakrN2006posteriori}.
Correspondingly, the piecewise constant interpolant of $U^{j}$  is denoted by $\overline U, $  
\begin{equation}\label{def_overU}
\overline U(t):= U^{n}, \quad t\in I_n\, .\end{equation}

 So now the discrete energy  $ \mathcal{G}_{IE, k}   $ can be written as follows
\begin{equation}\label{ParabFunctional_k}
\begin{split}
	\mathcal{G}_{IE, k} (U) = &\| \widehat U_t + L \overline U - \overline f \|^2_{L^2(0,T; L^2(\Omega))}  +\,     \int_{\Omega } |   \widehat U ^0  - u^0 |_{H^1(\Omega)}^2 \, \d  x\\
=& \int_0^T \|  \widehat U_t + L \overline U - \overline f \|^2_{L^2(\Omega)}\, \d t  +\,     \int_{\Omega } |   \widehat U^0  - u^0 |_{H^1(\Omega)}^2 \, \d  x\, .
\end{split}
\end{equation}

\subsubsection{Stability}

The stability of  $ \mathcal{G}_{IE, k} $ is proved next. 

\begin{proposition}\label{TD:EquicoercivityofG} The functional $ \mathcal{G}_{IE, k}$ defined in \eqref{ParabFunctional_k} is equicoercive with respect to $ \widehat U , \overline U $. That is,
\begin{equation}\label{GEquicoercivity_TD}
\begin{gathered}
\textrm{If} \;\; \mathcal{G}_{IE, k}(U) \leq C \;\; \textrm{for some} \;\: C >0 \;,\; \textrm{we have for $C_1$ independent of $U$ and $k$ that} \\
\|\overline U\|_{L^2(0,T; H^2(\Omega))} + \|\widehat U'\|_{L^2(0,T; L^2(\Omega))} \leq C_1\, .
\end{gathered}
\end{equation}
\end{proposition}

\begin{proof}
We have $\int_{\Omega_T} ( |  \widehat U_t + L \overline U |^2 -2 \overline f \: (  \widehat U_t + L \overline U   ) + |\overline f|^2 ) \leq C\, .$
Thus we can conclude that 
since $ f \in L^2(\Omega_T) $, we have the uniform bound $\| \widehat U_t + L \overline U \|_{L^2(0,T; L^2(\Omega))} \leq C_1\, . $
We shall need a discrete maximal regularity estimate in the present Hilbert-space setting. To this end we observe, 
\begin{equation} \label{dMR}
	\begin{split}
  \|\widehat U_t + L \overline U     \|^2_{L^2(0,T; L^2(\Omega))} 
& =  \|\widehat U_t     \|^2_{L^2(0,T; L^2(\Omega))}  
+  \| L \overline U     \|^2_{L^2(0,T; L^2(\Omega))} 
+2 \sum _{n=1}^N \, \int_{I_n}  \, 
  \<   \widehat U_t , L \overline U     \> \, \ d t 
\\
& =  \|\widehat U_t     \|^2_{L^2(0,T; L^2(\Omega))}  
+  \| L \overline U     \|^2_{L^2(0,T; L^2(\Omega))} \\
& \qquad +2 \sum _{n=1}^N \, \int_{I_n}  \,   \<    \big [ \frac{U ^n-  U^{n-1}}{k_n}  \big ], L   U ^n          \> \, \ d t \\
& =  \|\widehat U_t     \|^2_{L^2(0,T; L^2(\Omega))}  
+  \| L \overline U     \|^2_{L^2(0,T; L^2(\Omega))}  
+ \<   L\,  {U ^N }   ,  U ^N\>        \\
& \qquad + \sum _{n=1}^N \,    \,   \<  L\,   \big [  {U ^n-  U^{n-1}}   \big ],   {U ^n-  U^{n-1}  }       \> \,  -  \<   L\,  {U ^0 }   ,  U ^0\> \, .
 \end{split}
\end{equation}
Since all but the last term $\<   L\,  {U ^0 }   ,  U ^0\> $ are positive, we conclude, 
\begin{equation}
	\begin{split}
   \|\widehat U_t     \|^2_{L^2(0,T; L^2(\Omega))}  
+  \| L \overline U     \|^2_{L^2(0,T; L^2(\Omega))}  
\leq \|\widehat U_t + L \overline U     \|^2_{L^2(0,T; L^2(\Omega))}  +  \<   L\,  {U ^0 }   ,  U ^0\> \, ,
 \end{split}
\end{equation}
and the proof is complete.
\end{proof}

\subsubsection{$\liminf$ inequality}\label{TD:liminf}
  We assume there is a sequence, still denoted by $ U_\ell $, such that $ \mathcal{G}_{IE, \ell} (U_\ell) \leq C $ uniformly in $ \ell $, otherwise $  \liminf_{\ell \rightarrow \infty}\mathcal{G}_{IE, \ell} (U_\ell) = + \infty . $
From the discrete stability estimate, the uniform bound $ \mathcal{G}_{IE, \ell} (U_\ell) \leq C $ implies that $\|\overline U_\ell\|_{L^2(0,T; H^2(\Omega))} + \|\widehat U_\ell'\|_{L^2(0,T; L^2(\Omega))} \leq C_1\, ,$
 are uniformly bounded. By the relative compactness in $L^2(0,T; L^2(\Omega))$ we have (up to a subsequence not re-labeled)  the existence of $u_{(1)}$ and $u_{(2)}$ such that 
\begin{equation}\label{proofGlowersemicontinuTDiscr}
\begin{gathered}
L \overline  U_\ell \rightharpoonup L u_{(1)}\;\: \textrm{and} \;\: \widehat U_\ell' \rightharpoonup u_{(2)} '\;\: \textrm{weakly in} \;\: L^2(0,T; L^2(\Omega))\, .
\end{gathered}
\end{equation}
Notice that, for any space-time test function $\varphi\in C^\infty _0 $ there holds
(we have set $\tilde \varphi ^n :=  \frac 1 {k_n}\int_{I_n } \varphi     \, \ d t $)
\begin{equation}\label{IE_mr_estimate}
	\begin{split}
-	&\int_0^T  \<  \widehat U _\ell  , \varphi ' \>   \d t   = 
\int_0^T  \<  \widehat U_\ell ' , \varphi \>   \d t \\
& \ = 
\sum _{n=1}^N \, \int_{I_n}  \< \big [ \frac{U_\ell  ^n-  U_\ell ^{n-1}}{k_n}  \big ], \varphi \>  \, \ d t 
=	\sum _{n=1}^N \,    \<     U _\ell ^n, \tilde \varphi ^n\> -   \<  U_\ell  ^{n-1},  \tilde \varphi ^n \>  \, \\
& \ =	\sum _{n=1}^N \,    \<     U _\ell ^n,   \varphi ^{n-1}\> -   \<  U_\ell ^{n-1},    \varphi ^{n-1} \>  
+ \sum _{n=1}^N \,    \<     U_\ell  ^n, \big [ \tilde \varphi ^n  - \varphi ^{n-1} \big ]\> -   \<  U_\ell ^{n-1},  \big [ \tilde \varphi ^n  - \varphi ^{n-1} \big ]\> \\
& \ =	 - \int_0^T  \<  \overline U _\ell  , \varphi ' \>   \d t  
+ \sum _{n=1}^N \,    \<   \big [  U _\ell ^n -      U_\ell  ^{n-1}\big ], \big [ \tilde \varphi ^n  - \varphi ^{n-1} \big ]\> \, . \\
	\end{split}
\end{equation}
By the uniform bound, 
$$\|\widehat U_\ell'\|^2_{L^2(0,T; L^2(\Omega))}  
=  \sum _{n=1}^N \, \frac 1 {k_n}\|  U_\ell  ^n -      U_\ell ^{n-1}\| ^2 _{L^2(\Omega)} \leq C_1 ^2 \, ,$$
and standard approximation properties for 
$\tilde \varphi ^n  - \varphi ^{n-1}$ we conclude that for any fixed test function, 
\begin{equation}
	\begin{split}
 &\int_0^T  \<  \widehat U _\ell  , \varphi ' \>   \d t     - \int_0^T  \<  \overline U _\ell  , \varphi ' \>   \d t  \to 0, 
 \qquad \ell \to \infty\, . 
	\end{split}
\end{equation}
We can conclude therefore that  $u_{(1)} = u_{(2)} = u$ and thus,  $ \widehat U_\ell' + L \overline  U_\ell -  \overline  f \rightharpoonup u' + L u - f , $
$\ell \to \infty\, . $
%
The convexity of $ \int_{\Omega_T} | \cdot |^2 $ implies weak lower semicontinuity, that is
\begin{equation}\label{proofGlowersemicontinuityeq2}
\int_{\Omega_T} | u' + L u -f |^2  \leq \liminf_{\ell \rightarrow \infty}  \int_{\Omega_T} |  \widehat U_\ell' + L \overline  U_\ell -  \overline  f|^2
\end{equation}
and therefore we conclude that $ \mathcal{G}(u) \leq \liminf_{\ell \rightarrow \infty} \mathcal{G}_{IE, \ell} (U_\ell).   $

\subsubsection{$\limsup$ inequality}\label{TD:limsup}
Let $ w \in H^1(0,T; L^2(\Omega))\cap L^2(0,T; H^2(\Omega)  \cap  H^1_0(\Omega))    $ We will now show the existence of a recovery sequence $(w_\ell)$ such that $w_\ell \to w$ and 
   $ \mathcal{G}(w) = \lim_{\ell \rightarrow \infty} \mathcal{G}_{IE, \ell} (w_\ell) .$
 {By employing an one-dimensional (in the time variable) extension operator, as for instance the one used in the proof of \cite{Evans}[Theorem 1, Section 5.4], and using a regularisation with a smooth kernel,}
we can select a sequence $(w_\delta) \subset C^\infty(0,T; H^2(\Omega))$ with the properties
\begin{equation}
\begin{split}\label{w_delta_TD}
	\|&w -w_\delta\|_{H^1(0,T; L^2(\Omega))\cap L^2(0,T; H^2(\Omega))} \lesssim \delta \, , \quad \text{and,}\\
 |& w_\delta ' |_{H^1(0,T; L^2(\Omega))\cap L^2(0,T; H^2(\Omega))} \lesssim \frac{1}{\delta} \, .
	\end{split}
\end{equation}

If $w_{\delta, \ell}\in W_\ell$ is a neural network function satisfying \eqref{w_ell_7_TD}, \eqref{w_ell_7_hSm_TD},  we would like to show 
\begin{equation}\label{GL2conv}
||\widehat w_{\delta, \ell} ' + L \, \overline w_{\delta, \ell} -\overline f||_{L^2(0,T; L^2(\Omega))} \rightarrow ||w' + L w -f||_{L^2(0,T; L^2(\Omega))}
\end{equation}
where $\delta = \delta (\ell)$ is appropriately selected. 
Then,
\begin{equation}\label{GL2limit}
\mathcal{G}_{IE, \ell} (w_{\delta, \ell}) \rightarrow \mathcal{G}(w)\, .
\end{equation}
To this end it suffices to consider the difference
\begin{equation}\label{GL2conv_3}
\|\widehat w_{\delta, \ell} ' + L \, \overline w_{\delta, \ell} -  w' - L w  \|_{L^2(0,T; L^2(\Omega))} \, .
\end{equation}
We have 
\begin{equation}
	\begin{split}
 \|\widehat w_{\delta, \ell} '  + L \, \overline w_{\delta, \ell} -   w' - L w   \|_{L^2(0,T; L^2(\Omega))}
 \leq & \|\widehat w_{\delta, \ell} '  + L \, \overline w_{\delta, \ell} -   \widehat w_{\delta} '  -L \, \overline w_{\delta }   \|_{L^2(0,T; L^2(\Omega))}\\
&+ \|\widehat w_{\delta } ' + L \, \overline w_{\delta } -  w' -  L w  \|_{L^2(0,T; L^2(\Omega))}\\
= & : A_1 +A_2\, .
	\end{split}
\end{equation}
To estimate $A_1 $ we proceed as follows: Let $\theta _\ell (t) : = w_{\delta, \ell} (t)     -    w_{\delta}(t)\, .$ Then, 
\begin{equation}
	\begin{split}
  \|\widehat w_{\delta, \ell} '   & -   \widehat w_{\delta} '     \|^2_{L^2(0,T; L^2(\Omega))} 
= \sum _{n=1}^N \, \int_{I_n}  \,  \big \|  \frac{\theta  _\ell  ^n-  \theta  _\ell ^{n-1}}{k_n}  \big  \|^2 _{ L^2(\Omega)} \, \ d t \\
& = \sum _{n=1}^N \,   \, \frac 1 {k_n}  \big \| \int_{I_n} \theta  _\ell  ' (t)  \, \ d t \big  \|^2 _{ L^2(\Omega)}   \leq  \sum _{n=1}^N \,   \, \frac 1 {k_n}  \int_{I_n}  \big \| \theta  _\ell  ' (t) \big  \|^2 _{ L^2(\Omega)} \, \ d t \, k_n =   \|\theta  _\ell  '     \|^2_{L^2(0,T; L^2(\Omega))} \, .
	\end{split}
\end{equation}
Similarly, 
\begin{equation}
	\begin{split}
  \|L \, \overline w_{\delta, \ell}   & -L \, \overline w_{\delta }     \| _{L^2(0,T; L^2(\Omega))} 
  = \Big \{ \sum _{n=1}^N \, \int_{I_n}  \,  \big \| L \,   \theta  _\ell  ^n  \big  \|^2 _{ L^2(\Omega)} \, \ d t  \Big \}^{1/2} \\
  &\leq \Big \{ \sum _{n=1}^N \,k_n  \,  \big \| L \,   \theta  _\ell  ^n 
  	- \frac 1 {k_n}  \int_{I_n} L \,   \theta  _\ell  (t) \d t  \big  \|^2 _{ L^2(\Omega)} \,  \Big \}^{1/2}
  +\Big \{ \sum _{n=1}^N \,  \frac 1 {k_n} \,  \big \|      \int_{I_n} L \,   \theta  _\ell  (t) \d t \big  \|^2 _{ L^2(\Omega)} \,  \Big \}^{1/2}\\
  & \leq \Big \{ \sum _{n=1}^N \,k_n  \,  \big \| L \,   \theta  _\ell  ^n 
  	- \frac 1 {k_n}  \int_{I_n} L \,   \theta  _\ell  (t) \d t  \big  \|^2 _{ L^2(\Omega)} \,  \Big \}^{1/2}
  	+  \|L\, \theta  _\ell       \| _{L^2(0,T; L^2(\Omega))} \, .
	\end{split}
\end{equation}
It remains to estimate, 
\begin{equation}
	\begin{split}
 \Big \{ \sum _{n=1}^N \,k_n  \,  \big \| L \,   \theta  _\ell  ^n 
  &	- \frac 1 {k_n}  \int_{I_n} L \,   \theta  _\ell  (t) \d t  \big  \|^2 _{ L^2(\Omega)} \,  \Big \}^{1/2}
  =\Big \{ \sum _{n=1}^N \, \frac 1 {k_n}   \,  \big \|  \int_{I_n}\, \big [ L\,  \theta  _\ell  ^n - L \,   \theta  _\ell  (t) \big ]\, \d t  \big  \|^2 _{ L^2(\Omega)} \,  \Big \}^{1/2}\\
  &\leq 
   \Big \{ \sum _{n=1}^N \, \frac 1 {k_n}   \,  \big [  \int_{I_n}\, \big \| L\,  \theta  _\ell  ^n - L \,   \theta  _\ell  (t) \big \| _{ L^2(\Omega)}\, \d t  \big  ] ^2  \,  \Big \}^{1/2}    \\
    &\leq 
   \Big \{ \sum _{n=1}^N \, \frac 1 {k_n}   \,  \big [ \int_{I_n}\, \int_{I_n}\, \big \|  L \,   \theta  _\ell  ' (s) \big \| _{ L^2(\Omega)}\, \d s \, \d t  \big  ] ^2  \,  \Big \}^{1/2}    \\
   &=
   \Big \{ \sum _{n=1}^N \,   {k_n}   \,  \big [    \int_{I_n}\, \big \|  L \,   \theta  _\ell  ' (t) \big \| _{ L^2(\Omega)} \, \d t  \big  ] ^2  \,  \Big \}^{1/2}    \leq k\,   \|L\, \theta  _\ell  '     \| _{L^2(0,T; L^2(\Omega))} \, .
	\end{split}
\end{equation}
We conclude therefore that,  $k=\max_n k_n,$
\begin{equation}
	\begin{split}
A_2 \leq \|  \theta  _\ell  '     \| _{L^2(0,T; L^2(\Omega))}  + \|L\, \theta  _\ell       \| _{L^2(0,T; L^2(\Omega))} + k\,   \|L\, \theta  _\ell  '     \| _{L^2(0,T; L^2(\Omega))} \, . 
	\end{split}
\end{equation}
On the other hand, standard time interpolation estimates yield, 
  \begin{equation}
	\begin{split}
A_1 \leq C \, k \, \big [ 
\|  w_{\delta } '' \| _{L^2(0,T; L^2(\Omega))}  + \| L \,   w_{\delta}'       \|_{L^2(0,T; L^2(\Omega))}  \big ]\, .
	\end{split}
\end{equation}
Hence, we have using \eqref{w_ell_7_TD}, \eqref{w_ell_7_hSm_TD}, \eqref{w_delta_TD}, 
\begin{equation}
	\begin{split}
A_1+ A_2 
\leq  & \, \beta_\ell (w_\delta) + C \frac k 
{\delta ^{m+1} }  
+ C \frac k 
{\delta   }  
%
	\end{split}
\end{equation}
Therefore, we conclude that \eqref{GL2conv} holds upon selecting  $\delta = \delta (\ell, k)$   appropriately.   

\subsubsection{\it Convergence of the minimisers}
 In this subsection, we conclude the proof  that the sequence of discrete minimisers $ (u_\ell) $ converges in $L^2(0,T; H^1(\Omega))$ to the    minimiser of the continuous problem.  

\begin{theorem}[Convergence]\label{Thrm:TimeD} Let $ \mathcal{G},\; \mathcal{G}_{IE, \ell} $ be the energy functionals defined in \eqref{ParabFunctional} and \eqref{Functional_k} respectively. 
Let $u$ be the exact solution of \eqref{ParabolicPDE} and let $ (u_\ell) , $ $u_\ell \in V_{\ell}, $ be a sequence of minimisers of $ \mathcal{G}_{IE, \ell}  $, i.e.
\begin{equation}\label{minofTDEnergy}
\mathcal{G}_{IE, \ell} (u_\ell) = \inf_{v_\ell \in W_\ell} \mathcal{G}_{IE, \ell} (v_\ell)\, .
\end{equation}
Then, 
\begin{equation}\label{TDConvOfDiscrMinE}
\hat u_\ell \rightarrow u,  \;\; \; \textrm{in} \;\: L^2(0,T; H^1(\Omega))  ,\;\ \  
\end{equation}
where $\hat u_\ell $ is defined by \eqref{def_hatU}.
%
%
\end{theorem}
\begin{proof}
Next, let $u\in L^2(0,T; H^2(\Omega))\cap H^1(0,T; L^2(\Omega)) $ be the   solution of \eqref{ParabolicPDE}. 
Consider the sequence of minimisers $(u_\ell)\, .$ Obviously, 
$$ \mathcal{G}_{IE, \ell}  (u_\ell) \leq \mathcal{G}_{IE, \ell} (v_\ell), \qquad \text{for all } v_\ell \in V_\ell\, . $$
In particular,  
$$ \mathcal{G}_{IE, \ell}  (u_\ell) \leq \mathcal{G}_{IE, \ell}  (\tilde u_\ell),  $$
where $\tilde u_\ell $ is the recovery sequence $w_{\delta, \ell} $ corresponding to $w=u$ constructed  above. 
Hence, we conclude that the sequence $ \mathcal{G}_{IE, \ell}  (u_\ell)$ is   uniformly  bounded. The stability-equi-coercivity of the discrete functional, see Proposition \ref{TD:EquicoercivityofG},    implies   
  that
\begin{align}
\|\overline u_\ell\|_{L^2(0,T; H^2(\Omega))} +
\|\widehat u_\ell \|_{L^2(0,T; H^2(\Omega))} + \|\widehat u_\ell'\|_{L^2(0,T; L^2(\Omega))} \leq C  \, .
\end{align} 
  The  Aubin-Lions theorem ensures that there exists $   \tilde u\in L^2(0,T; H^1(\Omega)) $ 
such that $ \widehat u _\ell   \rightarrow   \tilde u$ in 
$L^2(0,T; H^1(\Omega))  $ up to a subsequence not re-labeled.
Furthermore the previous analysis shows that $L\tilde u \in L^2(0,T; L^2(\Omega))\, .  $
To prove that $   \tilde u $ is the minimiser of $ \mathcal{G} ,$ and hence $   \tilde u= u,$ we combine the 
results of Sections \ref{TD:liminf} and  \ref{TD:limsup}: Let   $ w \in H^1(0,T; L^2(\Omega))\cap L^2(0,T; H^2(\Omega)) . $ We did show the existence of a recovery sequence $(w_\ell)$ such that $w_\ell \to w$ and 
   $$ \mathcal{G}(w) = \lim_{\ell \rightarrow \infty} \mathcal{G}_{IE, \ell}  (w_\ell) .$$
   Therefore,  the $\liminf$ inequality 
and the fact that $ u_\ell $ are   minimisers of the discrete problems imply that
\begin{align}
 \mathcal{G} (\tilde u)  
 \le  \liminf_{\ell \rightarrow \infty}   \mathcal{G} _{IE, \ell}  (u_\ell)  
 \le  \limsup_{\ell \rightarrow \infty} 
 \mathcal{G}_{IE, \ell}  (u_\ell) 
 \le  \limsup_{\ell \rightarrow \infty} 
 \mathcal{G} _{IE, \ell}  (w_\ell)  =  
 \mathcal{G}  (w) , 
\end{align}
for all $ w \in H^1(0,T; L^2(\Omega))\cap L^2(0,T; H^2(\Omega)) . $  Therefore $  \tilde u $ is   the  minimiser of $  \mathcal{G} ,$ hence 
$   \tilde u= u$ and the entire sequence satisfies 
$$
\hat u_\ell \rightarrow u,  \;\; \; \textrm{in} \;\: L^2(0,T; H^1(\Omega))  .\;\ \  
$$
Therefore the proof is complete. \end{proof}

\subsubsection{\it Explicit time discrete training}
It will be interesting to consider a seemingly similar (from the point of view of quadrature and approximation) discrete functional:
\begin{equation}\label{Functional_k_Ex}
\mathcal{G}_{k, EE} (v) = \sum _{n=1}^N \,  k_n \, \int_{\Omega }\big | \frac{v ^n-  v^{n-1}}{k_n} + L v ^{n-1}  - f^{n-1} \big |^2 \, \, \d x +\,     \int_{\Omega } |  v  - u^0 |^2 d  x,
\end{equation}
and compare its properties to the functional $\mathcal{G}_{k, IE} (v) $ and the corresponding $V _{\mathcal{N}}$ minimisers. The functional \eqref{Functional_k_Ex}  is related to \emph{explicit} Euler discretisation in time as opposed to the \emph{implicit} Euler discretisation in time for $\mathcal{G}_{k, IE} (v).$ Clearly, in the discrete minimisation framework, both energies are fully implicit, since 
the evaluation of the minimisers involves the solution of global space-time 
problems. It is therefore rather interesting that these two energies result in completely different stability properties. 

Let us first note that it does not appear possible that a discrete coercivity such as \eqref{GEquicoercivity_TD} can be proved. Indeed, an argument similar to \eqref{dMR} is possible but with the crucial difference that the second to last term of this relation will be negative instead of positive. This is a fundamental point directly related to the (in)stability of the forward Euler method. Typically for finite difference forward Euler schemes one is required to assume a strong CFL condition of the 
form $k\leq C h^2 $ where $h $ is the spatial discretisation parameter
to preserve stability. It appears that a phenomenon of similar nature is present in our case as well. Although we do not show stability bounds when spatial training is taking place, the numerical experiments show that the stability behaviour of the explicit training method deteriorates when we increase the number of spatial training points while keeping $k$ constant. These stability considerations are verified by the numerical experiments we present below. Indeed, these computations provide convincing evidence that 
coercivity bounds similar to  \eqref{GEquicoercivity_TD} are necessary for  stable behaviour of the approximations.   In the computations we solve the one dimensional heat equation with zero boundary conditions  and two different initial values plotted in black. All runs were performed using the package \emph{DeepXDE}, \cite{Karniadakis:pinn:dxde:2021}, with random spatial training and constant time step. 

 \begingroup
\footnotesize

\begin{figure}[h]
\centering
\includegraphics[width=0.40\textwidth]{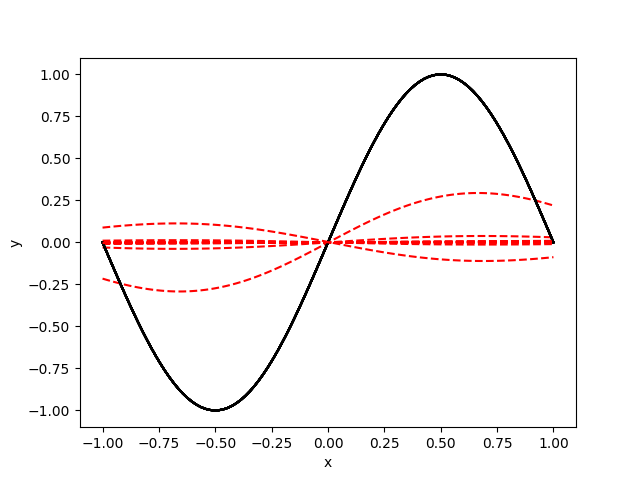}
\includegraphics[width=0.40\textwidth]{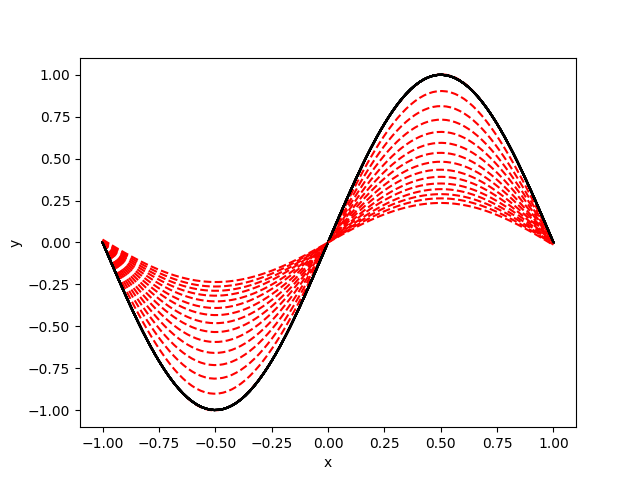}
\caption{\footnotesize\sl Explicit time discrete training. Left: time step  $0.4:$ the approximate solution  seems that diverge. 
 Right: time step  $0.01:$  with much smaller time step the approximate solution has stable behaviour.  }
\end{figure}
\begin{figure}[h]
\centering
\includegraphics[width=0.40\textwidth]{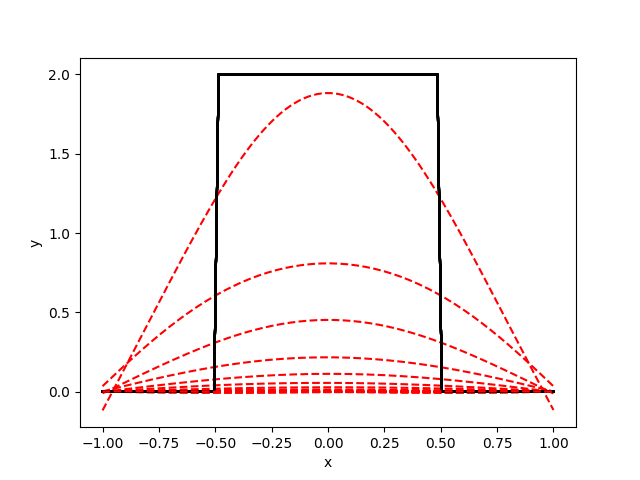}
\includegraphics[width=0.40\textwidth]{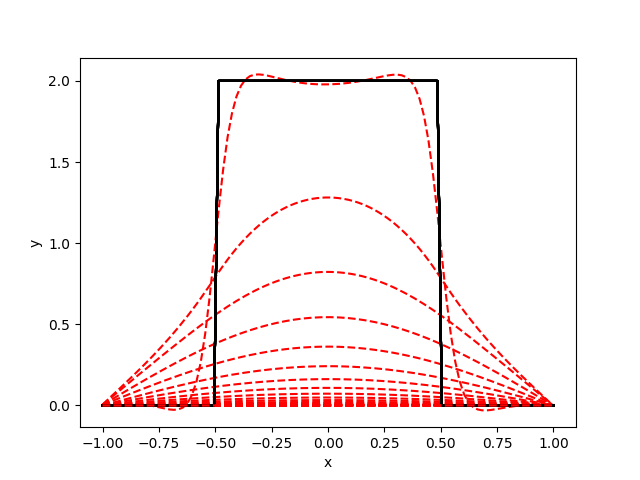}
\caption{\footnotesize\sl The approximations at times $t_n= n(0.2)$,  $n=1, 2, \dots , $ are displayed with red and the initial condition with black.  Left: Explicit time discrete training with time step  $0.2$ and $16$ training points.  The approximate solution  seems that diverge.  
Left: Implicit time discrete training with time step  $0.2$ and $16$ training points.   The approximate solution  seems reasonable.  }
\end{figure}
\begin{figure}[h]
\centering
\includegraphics[width=0.40\textwidth]{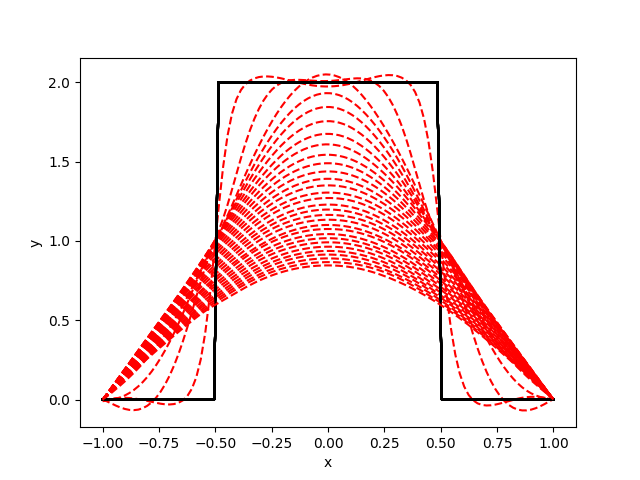}
\includegraphics[width=0.40\textwidth]{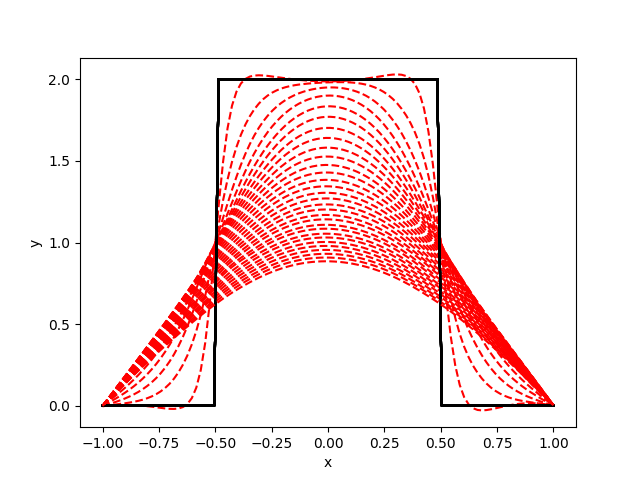}
\caption{\footnotesize\sl Left: Explicit time discrete training with time step  $0.01$ and $16$ training points.  The approximate solution  seems reasonable due to the much smaller time step. 
Left: Implicit time discrete training with time step  $0.01$ and $16$ training points.    }
\end{figure}
\begin{figure}[h]
\centering
\includegraphics[width=0.40\textwidth]{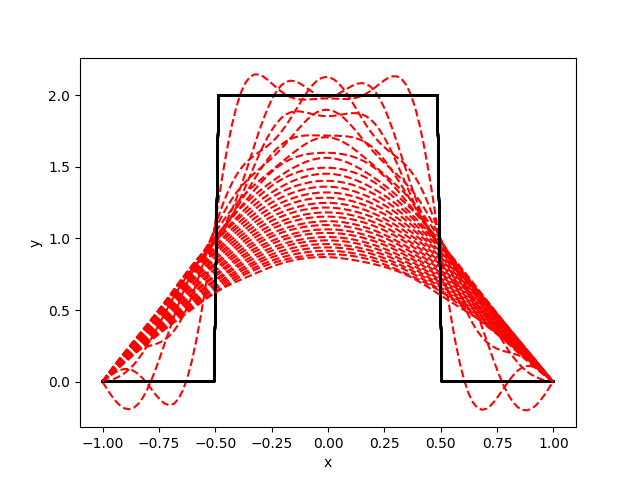}
\includegraphics[width=0.40\textwidth]{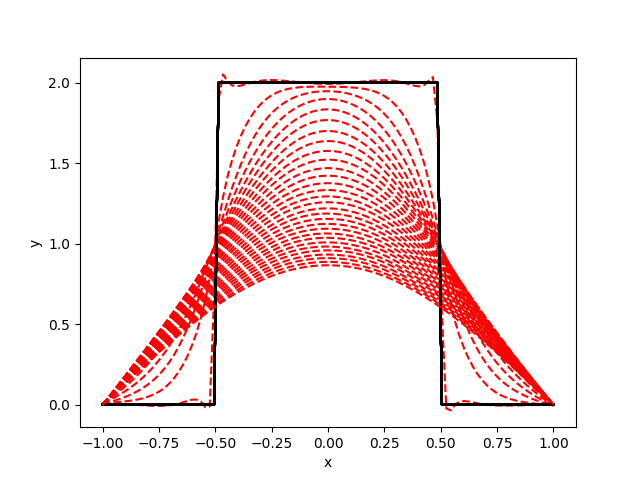}
\caption{\footnotesize\sl Left: Explicit time discrete training with time step  $0.01$ and $100$ training points.  Initial instabilities are again evident,  due to the higher number of spatial training points while the time step is the same as in Figure 3. 
Left: Implicit time discrete training with time step  $0.01$ and $100$ training points.    }
\end{figure}

\endgroup
 
\newpage
\noindent
 {\textbf {Acknowledgments. \ } We would like to express our sincere gratitude to a referee of this paper for their  suggestions. They   brought to our attention that a density argument employed in a previous version was not valid as stated and pointed out  Proposition 2.8 of \cite{Tran2024}.
We also want to thank  G. Savar\' e for his insightful suggestions regarding the properties of solutions to elliptic problems in Lipschitz domains and K.\ Koumatos for several related discussions.
Furthermore, we would like to thank G.\ Akrivis, E.\ Georgoulis, G.\ Karniadakis, T.\ Katsaounis,  M.\ Loulakis, P.\ Rosakis, A.\ Tzavaras, and J.\ Xu for their valuable discussions and suggestions.}

 \begingroup
\footnotesize
\bibliographystyle{abbrv}
\bibliography{ml_bibliography}

\endgroup
\end{document}